\documentclass[11pt]{amsart}
\usepackage[margin=1.2in]{geometry}
\usepackage[T1]{fontenc}
\usepackage{graphicx}
\usepackage{mathabx}
\usepackage{amsmath, amsthm, amssymb, mathtools, thmtools}
\usepackage{hyperref}
\usepackage[capitalise]{cleveref}
\usepackage{tikz}
\usepackage{xcolor}
\usepackage{nicefrac}
\usepackage{setspace}
\theoremstyle{plain}
\newtheorem{lemma}{Lemma}[section]
\newtheorem{prop}[lemma]{Proposition}
\newtheorem{theorem}[lemma]{Theorem}
\newtheorem{corollary}[lemma]{Corollary}
\newtheorem{question}[lemma]{Question}
\newtheorem{conjecture}[lemma]{Conjecture}

\usepackage{enumerate}
\theoremstyle{definition}
\newtheorem{definition}[lemma]{Definition}

\theoremstyle{remark}

\DeclarePairedDelimiter{\abs}{\lvert}{\rvert}
\DeclarePairedDelimiter{\floor}{\lfloor}{\rfloor}
\DeclarePairedDelimiter{\ceil}{\lceil}{\rceil}

\definecolor{mine}{RGB}{100, 120, 115}

\DeclareTextFontCommand{\cem}{\color{mine}\itshape}

\numberwithin{equation}{section}

\newenvironment{proofofclaim}{\begingroup\begin{proof}[Proof of claim]}{\end{proof}\endgroup}

\newtheorem{claim}{Claim}[lemma]

 \hypersetup{
     colorlinks=true,
     linkcolor=mine,
     filecolor=mine,
     citecolor = mine,      
     urlcolor=mine,
     }

\usepackage{dsfont}
\newcommand{\defn}[1]{{\color{mine}{\emph{#1}}}}
\newcommand{\eps}{\varepsilon}
\newcommand{\EE}{\mathds{E}}
\newcommand{\ZZ}{\mathds{Z}}

\newcommand{\RR}{\mathds{R}}
\DeclareMathOperator{\Bin}{Bin}
\DeclareMathOperator{\dist}{dist}
\DeclareMathOperator{\ex}{ex}
\DeclareMathOperator{\dd}{\hbox{-}}

\let\setminus=\smallsetminus
\let\phi=\varphi

\title{Dirac subgraphs of powers of cycles are Hamiltonian}

\author{Richard Lang}
\address[Richard Lang]{Departament de Matem\`atiques, Universitat Polit\`ecnica de Catalunya, Bar\-ce\-lo\-na, Spain and Centre de Recerca Matem\`atica, Barcelona, Spain}
\email{\texttt{richard.lang@upc.edu}}

\author{Alp M\"uyesser}
\address[Alp M\"uyesser]{University of Oxford, UK}
\email{\texttt{alp.muyesser@new.ox.ac.uk}}

\author{Mathias Schacht}
\address[Mathias Schacht]{Fachbereich Mathematik, Universität Hamburg, Hamburg, Germany}
\email{\texttt{schacht@math.uni-hamburg.de}}

\author{Carl Schildkraut}
\address[Carl Schildkraut]{Department of Mathematics, Stanford University, Stanford, US}
\email{\texttt{carlsch@stanford.edu}}

\begin{document}

\begin{abstract}
    We show that, for every $\eps>0$ and all sufficiently large $k$, any spanning subgraph of the $k$th power of a cycle with minimum degree at least $(1+\eps)k$ contains a Hamilton cycle. This asymptotically settles a conjecture of Espuny D\'iaz, Lichev, and Wesolek.
\end{abstract}


\maketitle

\section{Introduction}

The \textit{local resilience} of a graph with respect to a property $\mathcal{P}$  measures how much the graph needs to be changed from the perspective of some vertex before $\mathcal{P}$ no longer holds. As an example, consider Dirac's theorem, which states that any graph $G$ with minimum degree at least $\abs{V(G)}/2$ contains a Hamilton cycle. Dirac's theorem measures the local resilience of the complete graph with respect to Hamiltonicity: unless half the edges incident on some vertex are deleted, the graph remains Hamiltonian. 

Let us say that $H\subseteq G$ is an \defn{$\alpha$-subgraph} if $d_H(v)\geq\alpha d_G(v)$ for all $v\in V(G)$. With this notation, Dirac's theorem asserts that, for any $\eps>0$ and any $n$, any $(1/2+\eps)$-subgraph of $K_n$ is Hamiltonian. The threshold $1/2$ is a natural target while studying local resilience with respect to properties that imply connectivity because a $(1/2-\eps)$-subgraph that is disconnected can be found simply by considering a random cut (say, for regular graphs with sufficiently large degree). So, in a sense, graphs whose $(1/2+\eps)$-subgraphs are all Hamiltonian display as much local resilience for Hamiltonicity as one can expect. As another example, one may consider $K_{n,n}$, whose $(1/2+\eps)$-subgraphs are also Hamiltonian (this is a well-known extension of Dirac's theorem, see for example~\cite{moon1963hamiltonian}). 

This optimal $1/2$ target can be attained in several natural contexts where the host graph is significantly sparser than $K_n$ or $K_{n,n}$. For example, Lee and Sudakov \cite{lee2012dirac} showed that the random graph $G(n,p)$ inherits the local resilience of $K_n$ by establishing that any $(1/2+\eps)$-subgraph of $G(n,p)$ (where $p\gg \log n / n$) is Hamiltonian with high probability. Very recently, analogous results have been obtained for pseudorandom graphs \cite{draganic2025hamilton}. This latter result implies in particular that, given a group with $n$ elements, a typical generating set with $C\log n$ elements induces a Cayley graph all of whose $(1/2+\eps)$-subgraphs are Hamiltonian. 

On the other hand, there are natural classes of Hamiltonian Cayley graphs whose $(1/2+\eps)$-subgraphs are not all Hamiltonian; the complete tripartite graphs $K_{n,n,n}$ are one example.\footnote{Suppose $n$ is even. One may delete all edges between two subsets of size $>3n/4$ to ensure there exists no matching of size $n/2$ between a certain pair of parts, after which one can observe that there exists no perfect matching in the remaining graph, implying that the graph is not Hamiltonian.} A more extreme example is given by the hypercube, which can be disconnected by removing a perfect matching. Espuny D\'iaz, Lichev, and Wesolek \cite{diaz2024local}, in their work studying resilience of random geometric graphs (to which we will shortly return), put forth the following interesting conjecture. 
\begin{conjecture}\label{conj:mainconj}
    Consider the Cayley graph on $\ZZ/n\ZZ$ with generating set $\{-k,\ldots, k\}\setminus\{0\}$.
    Any subgraph of this Cayley graph with minimum degree $k+1$ has a Hamilton cycle.
\end{conjecture}
The Cayley graph above can also be described as $C_n^k$, the $k$th power of a cycle on $n$ vertices, obtained by adding edges between vertices of distance $d$ in the cycle $C_n$ for each $d\in [1,k]$. So, an equivalent phrasing of the conjecture is that, for any $\beta>1/2$, any $\beta$-subgraph of $C_n^k$ contains a Hamilton cycle.

\cref{conj:mainconj} generalizes Dirac's theorem, as can be seen by taking $k\geq (n-1)/2$: in this case $C_n^k$ is the complete graph $K_n$. If true, \cref{conj:mainconj} would show that $C_n^k$, like $G(n,p)$, inherits the local resilience of $K_n$ with respect to Hamiltonicity. Although $C_n^k$ and $G(n,p)$ both ``imitate'' $K_n$, they do so in fairly different ways. For example, $C_n^k$ is far from being a good expander on a global level, unlike $G(n,p)$. On the other hand, $C_n^k$ locally resembles a complete graph with respect to a one-dimensional geometry, whereas $G(n,p)$ does not typically contain unusually large dense subgraphs.

Our main result confirms \cref{conj:mainconj} asymptotically.
\begin{theorem}\label{thm:main}
    There is a constant $k_0$ such that, for every $\eps>0$ and all integers~$n$ and $k$ with 
    $n\geq k\geq \max\{\eps^{-121},k_0\}$, all $(1/2+\eps)$-subgraphs of $C_n^k$ are Hamiltonian.
\end{theorem}

\noindent The dependence of $k$ on $\eps$ means that \cref{thm:main} is slightly weaker than \cref{conj:mainconj} in the minimum degree required. It remains open to resolve \cref{conj:mainconj} in an exact form (even for sufficiently large $k$).

Although our \cref{thm:main} is along somewhat different lines than most prior work on resilience of Hamiltonicity, our work is linked to prior work (e.g., \cite{draganic2025hamilton,lee2012dirac}) in the sense that they belong to a common class of results that point towards a \textit{rough structure theory} of Hamiltonian graphs, which we now briefly explain. Hamiltonian graphs necessarily contain (fractional) perfect matchings and are connected. Of course, such easily checkable conditions cannot be sufficient to guarantee Hamiltonicity, due to NP-hardness of the Hamilton cycle problem.\footnote{The Petersen graph gives a simple example of a connected graph with a perfect matching but no Hamilton cycle.} However, it turns out that if connectivity and the existence of a fractional perfect matching hold \textit{robustly},\footnote{One reasonable interpretation of this asserts that the listed properties hold for typical induced subgraphs of density large enough to prevent isolated vertices; see~\cite{lang2023tiling}.} then the existence of a Hamilton cycle can be established in various settings. Each of \cite{lee2012dirac, draganic2025hamilton} and our \cref{thm:main} can be thought of in this general framework. In \cref{subsec:sketch}, we discuss our proof methodology, which is guided by this perspective. 

We also remark that in recent works of Lang~\cite{lang2023tiling} and Lang and Sanhueza-Matamala~\cite{lang2024hypergraph}, the rough structure theory for Hamiltonicity is developed for dense (hyper-)graphs, unifying a fair amount of previous literature. Although our methods here are different, our work could be loosely thought of as an extension of \cite{lang2024hypergraph} (in the simpler context of $2$-uniform graphs) from~$K_n$ to $C_n^k$. Put differently, our focus here is on extending the rough structure theory from dense graphs to graphs that are locally dense (but globally sparse). The power of a cycle~$C_n^k$ is just one example of a graph which is locally dense and globally sparse, and it remains an interesting direction to investigate higher-dimensional generalizations of \cref{conj:mainconj}. By this, we mean that $C_n^k$ has a simple $1$-dimensional geometry. What can we say about graphs with the locally dense geometry of a $2$-dimensional object? This brings us to the subject of random geometric graphs, which was the original motivation behind \cref{conj:mainconj}.

\subsection{Random geometric graphs}

\cref{conj:mainconj} was motivated by analogous questions concerning random geometric graphs (to our knowledge, first raised by Frieze \cite[Problem 43]{Frieze2019HamiltonRandomGraphsBibliography}). To explain the connection, let us define the random graph $T_d(n,r)$ as follows: sample $n$ random points of the $d$-dimensional torus $(\RR/\ZZ)^d$, and add an edge between any two vertices whose Euclidean distance is at most $r$. Below is an intriguing conjecture posed by Espuny D\'iaz, Lichev, and Wesolek \cite{diaz2024local}. 
\begin{conjecture}\label{conj:geometric}
For every $\eps\in (0,1/2]$ and integer $d\geq 1$, there is a constant $C$ such that for $r\geq C(\log n / n)^{1/d}$ such that $T_d(n,r)$ satisfies with high probability that all of its $(1/2+\eps)$-subgraphs are Hamiltonian.
\end{conjecture}
\cref{thm:main} confirms the above conjecture when $d=1$. This is because for the given range on $r$ in the conjecture, $T_1(n,r)$ can, with high probability, be ``sandwiched'' between~$C_n^{(1-\eps)r}$ and $C_n^{(1+\eps)r}$ (see \cite[Lemma 1.12]{diaz2024local}).\footnote{While finishing this paper, we found that \cref{conj:geometric} is false for $d=2$; the crucial construction was suggested to us by ChatGPT 5.5 Pro. We discuss this further in \cref{app:counterexample}. It remains an interesting problem to determine the correct threshold for local resilience for higher-dimensional geometric random graphs.}

Espuny D\'iaz, Lichev, and Wesolek~\cite{diaz2024local} also conjectured an analogue of \cref{conj:geometric} where~$(\RR/\ZZ)^d$ is replaced by $[0,1]^d$. When $d=1$, this conjecture reduces to investigating the local resilience of $P_n^k$, the $k$th power of an $n$-vertex path (by an appropriate analogue of~\cite[Lemma 1.12]{diaz2024local}), and different forms of this conjecture were investigated further in \cite{EspunyDiaz2026DiracBandwidth} by working with stronger assumptions on the degree of the vertices near the boundary of the path. We show that this latter conjecture is false, so the natural analogue of \cref{conj:geometric} fails for $[0,1]^d$ when $d=1$, for a specific choice of~$r$.

\begin{prop}\label{counterexample}
For some absolute constant $\eps>0$ and some positive integers $k$ and $n$, there exists a $(1/2+\eps)$-subgraph of $P_n^k$ which does not contain a perfect matching (or a Hamilton cycle).
\end{prop}
\begin{proof}

Let $n$ be a multiple of $4$, and take $k=\rho n$ for some absolute constant $\rho\in(1/2,3/4)$.
Denote by $A$ the $n/4+1$ vertices nearest one of the endpoints of $P_n^k$, denote by $C$ those $n/4+1$ vertices nearest the other endpoint, and denote by $B$ the remaining ($n/2-2$) vertices. Consider the spanning subgraph $G$ of $P_n^k$ consisting of all edges incident on $B$. 
It is easy to verify that, for some $\eps>0$ depending only on $\rho$, the graph $G$ is a $(1/2+\eps)$-subgraph of $P_n^k$. However, $A$ and $C$ are independent sets whose mutual neighborhood (which is $B$) is strictly smaller than $\abs{A\cup C}$, implying that the subgraph cannot contain a perfect matching.
\end{proof}

\subsection{Finding non-spanning subgraphs in \texorpdfstring{$C_n^k$}{Cnk}}\label{sec:other-q}

The previous sections of this introduction have focused solely on minimum-degree thresholds for Hamiltonicity.
We now turn to other questions within the power of a cycle.

Recall that any graph on $n$ vertices with minimum degree above $\frac{t-2}{t-1}$ must contain a $K_t$ as a subgraph.\footnote{This follows simply by the pigeonhole principle, and is best possible, as witnessed by a complete balanced $(t-1)$-partite graph.} Equivalently, to eliminate all copies of $K_t$ in a $K_n$, one must necessarily delete at least a $\frac{1}{t-1}$ fraction of the edges incident to some vertex of the $K_n$. It turns out that copies of $K_t$ are more fragile in $C_n^k$: one can delete at most a $\frac{1}{2(t-2)}$ fraction of the edges incident on vertices of $C_n^k$ and eliminate all copies of $K_t$ (for $t>3$, the latter fraction is strictly smaller). The following result formalizes this assertion, and also shows that it is best possible.

\begin{theorem}\label{thm:dirac-clique}
    For each integer $t\geq 3$, every subgraph of $C_n^k$ with minimum degree strictly exceeding $(2-\frac1{t-2})k+1$ contains $K_t$ as a subgraph.
    Moreover, this bound is tight under the mild divisibility conditions $(t-2)\mid k$ and $\frac k{t-2}\mid n$ and the mild size condition $n>2k+\frac{2k}{t-2}$.
\end{theorem}

Related results were obtained in the work of Espuny D\'iaz, Lichev, and Wesolek~\cite[Section~4]{diaz2024local} for cycles instead of cliques.

We additionally pose a Tur\'an-type problem relative to $C_n^k$ (which can be viewed as a \textit{global} resilience problem, see the discussion in \cite{sudakov2008local}).

\begin{definition}\label{def:extremal}
    Let $F$ and $G$ be graphs.
    The \emph{extremal number $\ex(G,F)$ of $F$ in $G$} is the maximum number of edges in an $F$-free subgraph of $G$.
    We define the \emph{extremal density of $F$ in powers of cycles} to be
    \[\pi_{\mathrm{cyc}}(F)=\limsup_{k\to\infty}\limsup_{n\to\infty}\frac{\ex(C_n^k,F)}{e(C_n^k)}.\]
\end{definition}

\noindent Thus, $\pi_{\mathrm{cyc}}(F)$ measures the maximum edge density of an $F$-free subgraph of $C_n^k$, in the regime where $k$ is large and $n$ is much larger.

The Erd\H{o}s--Stone--Simonovits theorem states that $\ex(K_n,F)=(1-\frac1{\chi(F)-1}+o(1))\binom n2$ as $n$ grows, where $\chi$ denotes the chromatic number.
Therefore, the extremal density of a graph $F$ relative to complete graphs is exactly $1-\frac1{\chi(F)-1}$.

Perhaps surprisingly, the extremal density of a graph in powers of cycles does not generally equal its extremal density in complete graphs. Indeed, we have the following result.

\begin{prop}\label{prop:turan-K3}
    It holds that
    \[0.585\approx 2-\sqrt2\leq \pi_{\mathrm{cyc}}(K_3)\leq\frac35=0.6.\]
    More generally, we have
    \[\pi_{\mathrm{cyc}}(K_t)\geq t-1-\sqrt{(t-2)^2+1}=1-\frac1{2(t-2)}+\frac1{8(t-2)^3}-O\left((t-2)^{-5}\right).\]
    In particular, for every $t\geq 3$ and for large $k$ and larger $n$, the densest $K_t$-free subgraph of $C_n^k$ is denser than that guaranteed by the Dirac-type result \cref{thm:dirac-clique}.
\end{prop}
We remark that the same lower bound as in \cref{prop:turan-K3} was obtained by Antoniuk and Reiher in a more general analytic setting \cite{AR24}.

\subsection{Potential generalizations, and some open problems}\label{subsec:poss-generalize}

We believe that locally dense, globally sparse graphs form a fruitful setting in which to consider extremal problems. 
We present a few problems along these lines, some of which we have previously hinted at. Two such problems arise naturally from our discussion, namely the full strength of \cref{conj:mainconj} and the determination of $\pi_{\mathrm{cyc}}(K_t)$ for each $t\geq 3$ (extending our \cref{prop:turan-K3}).

\subsubsection*{Resiliency after sub-sampling}

Although our \cref{thm:main} and the prior work of Lee and Sudakov \cite{lee2012dirac} are orthogonal generalizations of Dirac's theorem, they may admit a common generalization as follows.
\begin{conjecture}
     Let $\eps>0$. There exists $k$ large enough so that for any $n\geq k$, independently sampling each edge of $C_n^k$ with probability $\gg \log n/k$ yields, with high probability, a graph all of whose $(1/2+\eps)$-subgraphs are Hamiltonian.
\end{conjecture}

\subsubsection*{More spanning structures in $C_n^k$}

Beyond Hamilton cycles and perfect matchings, there is a large body of work focused on finding other spanning structures in graphs under various conditions.
Well-known examples of spanning structures include clique factors (disjoint cliques) and powers of Hamilton cycles.
Two classic results in this direction state that an $n$-vertex graph of minimum degree $(1-1/t) n$ contains a $K_t$-factor if $t$ divides $n$ (Hajnal and Szemer{\'e}di~\cite{HS70}) and the $(t+1)$st power of a Hamilton cycle for $n$ large enough (Koml{\'o}s, S{\'a}rk{\"o}zy, and Szemer{\'e}di~\cite{komlos1998proof}).
Both of these results generalize (consequences of) Dirac's theorem and it is therefore natural to ask whether this extends to the cyclical setting.

\begin{question}\label{qn:Kt-factor}
    Let $t\geq3$.
    For which $\alpha>0$ does it hold that, for sufficiently large $k$ and for $n\geq k$ a multiple of $t$, every $\alpha$-subgraph of $C_n^k$ contains a $K_t$-factor?
\end{question}

By modifying the construction for the lower bound of \cref{thm:dirac-clique} (see \cref{cor:dirac-factor-construction}), one can see that $\alpha$ must be at least $1-\frac1{2(t-1)}-o(1)$, which differs markedly from the bound $1-1/t$ in the complete setting when $t\geq 3$.
Of course, one can pose the same question also for powers of cycles (as guest graphs):

\begin{question}\label{qn:power-Ham}
    Let $t\geq2$.
    For which $\alpha>0$ does it hold that, for sufficiently large $k$, every $\alpha$-subgraph of $C_n^k$ contains the $t^{\mathrm{th}}$ power of a Hamilton cycle?
\end{question}

\section*{Acknowledgements}

We thank Tara Abrishami for valuable discussions in the early stages of the project. We are grateful to Alberto Espuny D\'{\i}az and Lyuben Lichev for bringing reference~\cite{AR24} for the lower bounds of \cref{prop:turan-K3} to our attention and for sharing the discussion of a related upper bound obtained in joint work with Marcos Kiwi and Dieter Mitsche.

This material is based upon work supported by the National Science Foundation under Grant No.~DMS-1928930, while the authors were in residence at the Simons Laufer Mathematical Sciences Institute in Berkeley, California, during the semester of Spring 2025. 

Richard Lang was supported by the Ramón y Cajal programme (RYC2022-038372-I) and by grant PID2023-147202NB-I00 funded by MICIU/AEI/10.13039/501100011033.

Carl Schildkraut is supported by a National Science Foundation Graduate Research Fellowship Program under Grant No.~DGE-2146755.

\section{Proof overview}

In this section, we outline the proof of \cref{thm:main}.
\crefrange{subsec:notation}{subsec:proof-of-thm} are devoted to this outline; we comment in \cref{subsec:generalization} on which steps in our outline generalize readily to other locally dense, globally sparse settings and which steps require new ideas to extend.

\subsection{Notation}\label{subsec:notation}
By \defn{$C_n^k$} we denote the $k$th power of the cycle~$C_n$. A spanning subgraph $G$ of~$C_n^k$ is \defn{$\alpha$-Dirac} if $\delta(G) \geq (1 + \alpha)k$, which is equivalent to $G$ being a $(1+\alpha)/2$-subgraph of~$C_n^k$. 
We will avoid the latter term in the remainder of the paper.
Throughout, we will always assume that $G$ is $\alpha$-Dirac unless otherwise stated.

For $x,y,z \in \RR$, we write $z \in x \pm y$ to indicate that $z$ is between $x-y$ and $x+y$.
We denote the disjoint union of sets $X$ and $Y$ by $X \sqcup Y$.
For vertices $u, v$ of $G \subset C_n^k$, we denote by~\defn{$\dist_C(u, v)$} the distance from $u$ to $v$ in $C_n$ and by~\defn{$\dist_G(u, v)$} the distance from $u$ to $v$ in~$G$. 

\begin{definition}\label{def:set-properties}
    Let $G$ be any (not necessarily $\alpha$-Dirac) subgraph of $C_n^k$.
    Given a subset $X\subseteq V(G)$, an interval $I$ of $C_n$, some $0<\eps<1$, and some positive real $r$, we say that 
    \begin{itemize}
        \item $X$ is \defn{$(\eps,r)$-sparse in $I$} if $\abs{I\cap X}\leq\eps\abs{I}+r\cdot k^{4/5}$;
        \item $X$ is \defn{$(\eps,r)$-dense in $I$} if $\abs{I\cap X}\geq\eps\abs{I}-r\cdot k^{4/5}$.
    \end{itemize}
    
    We say that $X$ is
    \begin{itemize}
        \item \defn{$(\eps,r)$-locally sparse} if $X$ is $(\eps,r)$-sparse in every interval of length at most $2k+1$;
        \item \defn{$(\eps,r)$-locally dense} if $X$ is $(\eps,r)$-dense in every interval of length at most $2k+1$;
        \item \defn{$(\eps,r)$-distributed} if $X$ is $(\eps,r)$-locally sparse, $(\eps,r)$-locally dense, and moreover
        \[ \abs{X\cap N_G(x)\cap N_G(y)} \in \eps\abs{N_G(x)\cap N_G(y)} \pm r\cdot k^{4/5}\]
        for every two (not necessarily distinct) vertices $x,y\in V(G)$.
    \end{itemize}
We will frequently omit $r$ from the above notions; here we understand $r=1$.
We will occasionally apply the notion of $\eps$-locally sparse to multisets of vertices. When we do so, expressions like $\abs{I\cap X}$ count elements of $X$ with multiplicity.
\end{definition}

\subsection{Proof sketch}\label{subsec:sketch}

Throughout, we will allow ourselves to assume that $k$ is larger (resp.\ $\alpha$ is smaller) than whichever absolute constant we require.
\par Our aim is to find a Hamilton cycle in $G$, and we will follow an absorption strategy. The properties of the absorber we require are enumerated in the following definition.

\begin{definition}
    Let $G$ be a graph and let $S\subset V(G)$. 
    Let $x,y\in V(G)\setminus S$ be two distinct vertices.
    A set $A\subset V(G)\setminus S$ with $\{x,y\} \subset A$ is an \defn{$(x,y;S)$-absorber} if, for every subset $S'\subset S$, the graph $G[S'\sqcup A]$ has a spanning path with endpoints $x$ and $y$.
\end{definition}

We now give a sketch of the proof of \cref{thm:main}, stating the main ingredients which will be proven in later sections. In this sketch, we omit all mention of how the density parameters in different lemmas relate to each other.

We begin by isolating two distributed sets $S$ and $R$, which is possible using the following (more general) lemma.

\begin{lemma}[Partitioning a distributed set]\label{lem:distributed-partition}
    Let $G$ be any subgraph of $C_n^k$.
    Let $S$ be an $(\eps,r)$-distributed set for $G$, and let $\eps=\eps_1+\eps_2$. 
    Suppose that $k$ is larger than some absolute constant and that $r\min(\eps_1,\eps_2)\geq k^{-1/5}\eps$.
    Then $S$ can be partitioned into an $(\eps_1,r)$-distributed set and an $(\eps_2,r)$-distributed set.
\end{lemma}

Next, we find an $S$-absorber $A$ in $R$ via the following lemma.

\begin{lemma}[Absorber construction]\label{lem:absorber-constr}
    Let $G$ be an $\alpha$-Dirac subgraph of $C_n^k$.
    Let $R\subset V(C_n)$ be $\eps$-distributed, and let $S\subset V(C_n)\setminus R$ be $\eps'$-distributed.
    Suppose that $\alpha^{-1}k^{-1/5}\leq\eps'\leq10^{-22}\alpha^8\eps$.    
    Then there exist $x,y\in R$ for which $R$ contains an $(x,y;S)$-absorber $A$ in $G$.
\end{lemma}

We remark that in \cref{lem:absorber-constr}, we only need $S$ to be $\eps'$-locally sparse, but the stronger condition simplifies the proof.

We now release the vertices of $R\setminus A$ back to the rest of the graph, which we now call $G'$. As $S\sqcup A$ is fairly sparse, the removal of these vertices does not change the structure of our graph too much.

\begin{lemma}[Removing a locally sparse subset]\label{lem:loc-sparse-removal}
    Let $G=(V,E)$ be an $\alpha$-Dirac subgraph of~$C_n^k$ and let $J \subseteq V$ be an $\eps$-locally sparse set.
    Set $n' \coloneqq \abs{V} - \abs{J}$ and $\alpha' \coloneqq \alpha - 3\eps - k^{-1/5}$. 
    Then, the induced subgraph $G':=G[V\setminus J]$ is an $\alpha'$-Dirac subgraph of $C_{n'}^{k}$.
    Moreover, if~$S\subset V(G')$ is an $(\eps',r)$-locally sparse set in $G'$ for some $(\eps',r)$, then $S$ is $(\eps',r)$-locally sparse in $G$.
\end{lemma}

The next lemma asserts that we can vertex-decompose $G'$ into a number of paths whose endpoints are locally sparse. This is perhaps the most critical component in our proof.  
Recall that a \defn{linear forest} is a collection of pairwise disjoint paths.

\begin{lemma}[Path decomposition]\label{lem:path-decomp}
    Suppose that $k$ is sufficiently large.
    Let $G$ be an $\alpha$-Dirac subgraph of $C_n^k$ for ${\alpha>k^{-1/10}}$. 
    Then $G$ has a spanning linear forest whose multiset $U$ of endpoints\footnote{By ``multiset of endpoints,'' we mean that if a component of the spanning linear forest is a single vertex, then that vertex is counted twice.} is $(10^4k^{-1/12},1/2)$-locally sparse.
\end{lemma}

Next, we must decide how to connect up these endpoints of the paths found in the previous lemma. If the underlying host graph were a complete graph, this would not be a concern: any pairing of the endpoints that yields a Hamilton cycle when merged with the given paths would do the job. However, it is costly for us to pair vertices which are far away from each other on the cycle, as many vertices of $C_n^k$ must be used to connect two such vertices by a path. So, we must choose our pairing with some care. The following lemma gives us a pairing that is sufficient for our purposes.

\begin{lemma}[Path ordering]\label{lem:path-order}
    Let $\sigma$ be an involution on some set $S\subset V(C_n)$, and let $S'$ be the multiset in which every fixed point of $\sigma$ appears twice and every other element of $S$ appears once.
    
    Then there exists an ordering $v_1,w_1,\ldots,v_s,w_s$ of the elements of $S'$ and some selection $t\in\{\mathrm{clockwise},\mathrm{counterclockwise}\}^s$ such that $\sigma(v_i)=w_i$ for each $1\leq i\leq s$ and the following holds.
    For each $1\leq i\leq s$, consider the interval $I_i:=[w_i,v_{i+1}]$ (with $v_{s+1}=v_1$) in $C_n$ oriented clockwise or counterclockwise depending on $t_i$.
    Then each $x\in C_n$ lies in at most twelve $I_i$.
\end{lemma}

Given this order, we connect the paths, including the ``path'' $x\to y$ that we will choose within our absorber $A$, using vertices in $S$. 

\begin{lemma}[Path connection]\label{lem:path-connect}
    Let $S$ be an $\eps$-distributed set.
    Let $T$ be an $\eps'$-locally sparse multiset, disjoint from $S$, of some even size $2t$.
    Suppose that $T$ can be enumerated as $v_1,w_1,\ldots,v_t,w_t$ in such a way that the clockwise intervals $I_i:=[v_i,w_i]$ in total cover each point of $C_n$ at most $M$ times.

    Suppose that $M<\frac{\alpha\eps'}{12}k$ and $\alpha^{-1}k^{-1/5}\leq\eps'\leq10^{-3}\alpha\eps$.
    
    Then there exists, for each $1\leq i\leq t$, a path $P_i=(v_i,s_{1,i},\ldots,s_{\ell_i,i},w_i)$ of even length $\ell_i\leq 6+12\alpha^{-1}k^{-1}\dist_C(v_i,w_i)$ in $G$ whose internal vertices all lie in $S$. 
    Moreover, we can ensure that the $P_i$ are internally vertex-disjoint.
\end{lemma}

Neither the requirement of even paths nor the ability to take $M$ rather large is necessary to deduce the main theorem from the lemmas presented in this section. However, we will use \cref{lem:path-connect} (via a corollary, \cref{cor:path-connect}, presented in the next section) to prove \cref{lem:absorber-constr}, and here both properties are crucial.

Finally, we absorb the vertices in $S$ not used by the previous lemma into our $x\to y$ path via our absorber $A$. The result is a Hamilton cycle in $G$.

\vspace{2mm}

Before continuing, we provide a few comments on \cref{lem:path-decomp}. Firstly, a consequence of this lemma is that $\alpha$-Dirac subgraphs $G\subseteq C_n^k$ contain nearly spanning matchings, which is already a non-obvious statement. In fact, our proof of \cref{lem:path-decomp} works by first proving a version of this latter statement, and then bootstrapping this statement to find spanning subgraphs of $G$ that are regular and also relatively dense. Such subgraphs are helpful because regular dense graphs can be vertex-partitioned into a few paths (this is the content of the well-studied Magnant--Martin conjecture, see for example \cite{montgomery2025approximate, christoph2025new}). Our proof mimics a usual strategy to prove an approximate result of this form (see~\cite{montgomery2025approximate}). We begin by partitioning, using the Lov\'asz local lemma, the vertex set into small clusters whose pairwise combinations form nearly regular bipartite graphs, and hence contain nearly spanning matchings by an extension of Hall's marriage theorem. Combining matchings across pairwise combinations of consecutive clusters along a linear order (see \cref{fig:matchings-to-forest}), we obtain a collection of paths (including, potentially, some isolated vertices) spanning the vertex set.

\begin{figure}
\begin{tikzpicture}[scale=0.6, every node/.style={circle,fill,inner sep=1.6pt}]

\def\dx{3.5}
\def\halfedgelen{1.75}

\def\ybot{-2}
\def\ytop{2}

\foreach \i in {0,...,4}
{
  \pgfmathsetmacro{\y}{\ybot + (\ytop-\ybot)*\i/4}
  \node (A\i) at (0,\y) {};
}
\draw (0,0) ellipse (0.6 and 2.6);

\foreach \i in {0,...,4}
{
  \pgfmathsetmacro{\y}{\ybot + (\ytop-\ybot)*\i/4}
  \node (B\i) at (\dx,\y) {};
}
\draw (\dx,0) ellipse (0.6 and 2.6);

\foreach \i in {0,...,3}
{
  \pgfmathsetmacro{\y}{\ybot + (\ytop-\ybot)*\i/3}
  \node (C\i) at (2*\dx,\y) {};
}
\draw (2*\dx,0) ellipse (0.6 and 2.6);

\node[draw=none,fill=none] at (2.75*\dx,0) {$\cdots$};

\foreach \i in {0,...,3}
{
  \pgfmathsetmacro{\y}{\ybot + (\ytop-\ybot)*\i/3}
  \node (D\i) at (3.5*\dx,\y) {};
}
\draw (3.5*\dx,0) ellipse (0.6 and 2.6);

\foreach \i in {0,...,4}
{
  \pgfmathsetmacro{\y}{\ybot + (\ytop-\ybot)*\i/4}
  \node (E\i) at (4.5*\dx,\y) {};
}
\draw (4.5*\dx,0) ellipse (0.6 and 2.6);

\foreach \i in {0,1,2,3}
  \draw (A\i) -- (B\i);

\draw (B1) -- (C0);
\draw (B2) -- (C1);
\draw (B3) -- (C2);
\draw (B4) -- (C3);

\draw (D0) -- (E1);
\draw (D1) -- (E2);
\draw (D2) -- (E3);
\draw (D3) -- (E4);

\foreach \i in {0,...,3}
  \draw (C\i) -- ++(\halfedgelen,0);

\foreach \i in {0,...,3}
  \draw (D\i) -- ++(-\halfedgelen,0);

\end{tikzpicture}
\caption{Stitching nearly-spanning matchings into a spanning linear forest with few paths.}
\label{fig:matchings-to-forest}
\end{figure}
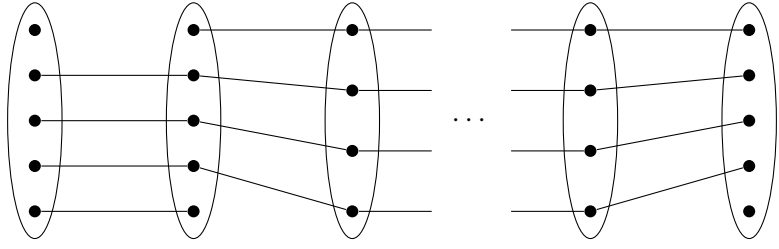

The utility of this partition into a few paths lies in a later step in which we stitch the said paths into a Hamilton cycle (facilitated by the absorber construction). The fewer paths there are, the easier this latter task is. It is not sufficient, however, that the number of paths to stitch up is small in aggregate --- it is also critical that the endpoints of these paths do not cluster in any small interval of $C_n^k$. To ensure this, we find each of our nearly spanning matchings between a consecutive pair of clusters via a random process, analyzed by the semi-random method (or R\"odl nibble; see \cite{alon2016probabilistic}).

\subsection{Proof of the main result}\label{subsec:proof-of-thm}

Following the proof sketch given in the previous subsection, we prove \cref{thm:main} assuming \crefrange{lem:distributed-partition}{lem:path-connect}.

\begin{proof}[Proof of \cref{thm:main}]
    Let $\eps'':=10^4k^{-1/12}$ and $\eps':=10^3\alpha^{-1}\eps''$ and $\eps=10^{22}\alpha^{-8}\eps'$.
    We note that, since $k$ is sufficiently large and $\alpha\geq k^{-1/121}$, we have
    \[\eps=10^{29}\alpha^{-9}k^{-1/12}\leq\alpha/10.\]

    Note that $V(C_n)$ is $1$-distributed.
    We first use \cref{lem:distributed-partition} twice to partition $V(C_n)$ into three sets $X\sqcup R\sqcup S$ where $R$ is $\eps$-distributed and $S$ is $\eps'$-distributed.
    By \cref{lem:absorber-constr}, there exist some $x,y\in R$ and some $(x,y;S)$-absorber $A\ni x,y$ in $R$.
    Let $B=R\setminus A$ and $Y=X\sqcup B$.
    Set $\alpha':=\alpha-3\eps-k^{-1/5}>k^{-1/10}$ and $n'=\abs{Y}$. 
    By \cref{lem:loc-sparse-removal}, the graph $G':=G[Y]$ is an $\alpha'$-Dirac subgraph of $C_{n'}^k$. 
    Use \cref{lem:path-decomp} to find a spanning linear forest in $G'$ whose multiset $U_0$ of endpoints is $(\eps'',1/2)$-locally sparse in $G'$. 
    By \cref{lem:loc-sparse-removal}, the set $U_0$ is also $(\eps'',1/2)$-locally sparse in $G$.
    
    Adjoin $x$ and $y$ to $U_0$ to form $U$.
    Let $\sigma\colon U\to U$ be the involution mapping each endpoint of each path to the other endpoint, and additionally swapping $x$ and $y$.
    The set $U$ is $\eps''$-locally sparse.
    By \cref{lem:path-order}, there exists an ordering $v_1,w_1,\ldots,v_t,w_t$ of $U$ for which $v_i=\sigma(w_i)$ for each $i$ and for which the intervals $I_i$ connecting $w_i$ and $v_{i+1}$ oriented suitably cover each vertex of $C_n$ at most $12$ times. 
    Rotate this ordering in such a way that $v_1=x$ and $w_1=y$. 
    For each $2\leq i\leq t$, let $P_i$ be the path in $Y$ connecting $v_i$ to $w_i$.
    
    By \cref{lem:path-connect} with $M=12$, since $U$ is $\eps''$-locally sparse and $S$ is $\eps'$-distributed, we may find internally vertex-disjoint paths $Q_i$ to connect each $w_i$ to each $v_{i+1}$ all of whose internal vertices lie in $S$.
    Let $S'$ be the set of vertices in $S$ covered by these paths.

    The concatenation of the $P_i$ and the $Q_i$ forms a path
    \begin{equation}\label{eq:nearly-whole-path}
        y\xrightarrow{Q_1}v_2\xrightarrow{P_2}w_2\xrightarrow{Q_2}v_3\to\cdots\to w_{t-1}\xrightarrow{Q_{t-1}}v_t\xrightarrow{P_t}w_t\xrightarrow{Q_t}x,
    \end{equation}
    which covers exactly the vertex set $Y\sqcup S'$.
    Since $A$ is an $(x,y;S)$-absorber, there exists an~$x\to y$ path which spans exactly the vertex set $A\sqcup (S\setminus S')$. Concatenating this path with \eqref{eq:nearly-whole-path} gives a Hamilton cycle of $G$, as desired.
\end{proof}

The next five sections will be devoted to the proofs of \crefrange{lem:distributed-partition}{lem:path-connect}. 

In \cref{sec:lemmas}, we prove \cref{lem:distributed-partition,lem:loc-sparse-removal} and deduce from \cref{lem:path-connect} a corollary which is useful in the proof of \cref{lem:absorber-constr}.
In \cref{sec:path-connection}, we prove \cref{lem:path-order,lem:path-connect}.
In \cref{sec:absorber-constr}, we prove \cref{lem:absorber-constr}.
Finally, we prove \cref{lem:path-decomp} in \cref{sec:Hall,sec:path-decomp}.

\subsection{Potential for generalization}\label{subsec:generalization}

In many places throughout the proof, we use structural properties of $C_n^k$ to simplify matters. 
Many of our arguments, however, extend without significant issue to other locally dense, globally sparse settings, such as the Cayley graph~$G$ on $(\ZZ/n\ZZ)^2$ consisting of vectors of length at most $k$. For example, in $G$, disks of radii at most $k$ can play the role of the intervals in our \cref{def:set-properties} and applications of the local lemma go through as the number of disks containing a given vertex grows only polynomially in $k$. An important input to our proof is that vertices at distance~$o(k)$ in the base cycle of $C_n^k$ share $\Omega(k)$ common neighbors in any $(1/2+\eps)$-subgraph of $C_n^k$ (by a simple application of the pigeonhole principle), supplying robust local connectivity. This same property holds for $G$ and, more generally, for any graph defined on a metric space where vertices are adjacent if and only if they have distance at most $k$.

\section{General lemmas}\label{sec:lemmas}

We begin by presenting a few additional lemmas which we will use throughout the proof.

The first lemma is the main way in which we access the property that $G$ is $\alpha$-Dirac. 
Repeated applications of this lemma allow us to construct, for example, an odd cycle in $G$ through an arbitrary vertex.

\begin{lemma}\label{lem:nbhd-inter}
    Let $x$ and $y$ be vertices of $G$.
    If $\dist_C(x,y)=d<2\alpha k$, then $x$ and $y$ have at least $2\alpha k-d-1$ common neighbors in $G$.
\end{lemma}
\begin{proof}
    Let $I_x$ and $I_y$ be intervals of length $2k+1$ centered around $x$ and $y$, respectively.
    We have $N_G(x)\subset I_x$ and $N_G(y)\subset I_y$. Therefore,
    \begin{align*}
        \abs{N_G(x)\cap N_G(y)}
        &\geq \abs{I_x\cap I_y}-\abs{I_x\setminus N_G(x)}-\abs{I_y\setminus N_G(y)}\\
        &=2k+1-d-(2k+1-\deg_G(x))-(2k+1-\deg_G(y))\\
        &\geq 2(1+\alpha)k-(2k+1)-d\\
        &=2\alpha k-d-1.\qedhere
    \end{align*}
\end{proof}

Next, we show \cref{lem:distributed-partition}, which states that well-distributed sets can be partitioned into well-distributed sets. The proof is a simple application of the Lov\'asz local lemma.

\begin{proof}[Proof of \cref{lem:distributed-partition}]
    We apply the Lov\'asz local lemma.
    Consider a procedure wherein a vertex in $S$ is allocated to $S_1$ with probability $\eps_1/\eps$ and to $S_2$ otherwise.
    Let $\mathcal T$ be a collection of subsets of $V(G)$ defined to contain exactly the intervals of length at most $2k+1$, as well as the sets $N_G(x)\cap N_G(y)$ for each pair $x,y\in V(G)$. For each $T\in\mathcal T$ and $j\in\{1,2\}$, let $A_{T,j}$ be the event that $\abs{S_j\cap T}$ and $\eps_j\abs{T}$ differ by more than $rk^{4/5}$.

    We note the following:
    \begin{enumerate}
        \item If no event $A_{T,j}$ occurs, then $S_1$ and $S_2$ are $\eps_1$-distributed and $\eps_2$-distributed, respectively.

        \item Each set $T\in\mathcal T$ is of size at most $2k+1$.

        \item Since $N_G(x)\cap N_G(y)$ is only nonempty when $\dist_C(x,y)\leq 2k$, each vertex $v\in V(G)$ lies in at most $2\binom{2k+1}2\leq 5k^2$ sets $T\in\mathcal T$.
    \end{enumerate}
    Properties (2) and (3) imply that each event $A_{T,j}$ is independent of all but at most $21k^3$ other events.
    Fix $j=1$ and $T$ for the moment.
    The size $X=\abs{S_1\cap T}$ has distribution
    $\Bin\left(\abs{S\cap T},{\eps_1}/{\eps}\right)$
    with mean $\mu=\frac{\eps_1}{\eps}\abs{S\cap T}$. Since $S$ is $(\eps,r)$-distributed, we have
    \[\abs[\big]{\mu-\eps_1\abs{T}}=\frac{\eps_1}{\eps}\abs[\big]{\abs{S\cap T}-\eps\abs{T}}\leq\frac{\eps_1}{\eps}rk^{4/5}.\]
    Therefore, 
    it follows by Chernoff's bound that
    \begin{align*}
    \Pr\big[A_{T,1}\big]
    &\leq  \Pr\left[X \notin \eps_1\abs{T} \pm rk^{4/5}   \right]\\
    &\leq\Pr\left[\abs{X-\mu}\geq rk^{4/5}-\abs[\big]{\mu-\eps_1\abs{T}}\right]\\
    &\leq\Pr\left[\abs{X-\mu}\geq rk^{4/5}\left(1-\frac{\eps_1}{\eps}\right)\right]\\
    &\leq 2\exp\left(-\frac{r^2k^{8/5}(\eps_2/\eps)^2}{3\mu}\right)\leq 2\exp\left(-\frac{r^2k^{3/5}\eps_2^2}{6\eps_1\eps}\right) \\
    &\leq 2\exp\left(-\frac{k^{1/5}\eps^2}{6\eps_1\eps}\right) \leq 2\exp\left(-\frac{k^{1/5} }{6}\right).
    \end{align*}
    Our assumptions on $\eps_1$, $\eps_2$, and $k$ imply that this probability is at most $1/(60k^3)$. The result now follows from the Lov\'asz local lemma.
\end{proof}

We will also frequently use the following corollary of \cref{lem:path-connect} (which is not proved until later on in the paper).

\begin{corollary}\label{cor:path-connect}
    Let $S$ be an $\eps$-distributed set.
    Let $T$ be an $\eps'$-locally sparse set, disjoint from~$S$, of some even size $2t$.
    Suppose that $T$ can be enumerated as $v_1,w_1,\ldots,v_t,w_t$ in such a way that $\dist_C(v_i,w_i)\leq L\cdot k$ for each $i$.

    Suppose that $L\geq 1$ and $\alpha^{-1}k^{-1/5}<\eps'<10^{-5}L^{-1}\alpha^2\eps$.

    Then there exists, for each $1\leq i\leq t$, some path $P_i=(v_i,s_{1,i},\ldots,s_{\ell_i,i},w_i)$ of even length $\ell_i\leq 6+12\alpha^{-1}k^{-1}\dist_C(v_i,w_i)$ in $G$ whose internal vertices all lie in $S$. Moreover, we can ensure that the $P_i$ are internally vertex-disjoint.
\end{corollary}

\begin{proof}
    Let $\eps'':=100\alpha^{-1}\eps'L$; we will apply \cref{lem:path-connect} with $\eps''$ instead of $\eps'$. 
    Note that $\alpha^{-1}k^{-1/5}<\eps'<\eps'' \leq \frac1{1000}\alpha\eps$. 
    In particular, $T$ is $\eps''$-locally sparse. 
    It thus suffices to show that, if the clockwise intervals connecting each $v_i$ and $w_i$ are drawn, then they cover each point fewer than $\frac{\alpha\eps''}{12}k$ times.

    Indeed, take any $z\in V(C_n)$. If $z$ is covered by some interval $[v_i,w_i]$, then in particular we have $\dist_C(z,v_i)\leq L\cdot k$. Since $T$ is $\eps'$-locally sparse, there are at most
    \[\eps'(2Lk+1)+\left(\frac{2Lk+1}{2k+1}+1\right)k^{4/5}\]
    values of $i$ which satisfy this. The result follows from
    \[\eps'(2Lk+1)+\left(\frac{2Lk+1}{2k+1}+1\right)k^{4/5}\leq 3\eps'Lk+3Lk^{4/5}\leq 6\eps'Lk<\frac{\alpha\eps''}{12}k.\qedhere\]
\end{proof}

We conclude the section by giving a quick proof of \cref{lem:loc-sparse-removal}.

\begin{proof}[Proof of \cref{lem:loc-sparse-removal}]
    It is clear that $G'$ can be viewed as a subgraph of $C_{n'}^k$.
    For any vertex $x\in V(G)\setminus J$, we have
    \[\deg_{G'}(x)=\deg_G(x)-\abs{N_G(x)\cap J}.\]
    Since $N_G(x)$ is contained in an interval of length $2k+1$ in $G$, we thus have
    \[\deg_{G'}(x)\geq(1+\alpha)k-\left(\eps(2k+1)+k^{4/5}\right)\geq k+\left(\alpha-3\eps-k^{-1/5}\right)k.\]
    Now, let $I$ be an interval of length $2k+1$ in $G$, and let $I'=I\cap V(G')$; this is an interval of length at most $2k+1$ in $G'$. We conclude that
    \[\abs{S\cap I}=\abs{S\cap I'}\leq \eps'\abs{I'}+rk^{4/5}\leq\eps'\abs{I}+rk^{4/5}.\qedhere\]
\end{proof}

\section{Connecting paths}\label{sec:path-connection}

We now begin proving the central lemmas in our proof of \cref{thm:main}.
We will do this roughly in reverse order, as the proofs of \cref{lem:absorber-constr,lem:path-decomp} are rather more technical than the others. Moreover, the proof of \cref{lem:absorber-constr} will use \cref{lem:path-connect} (via \cref{cor:path-connect}) a few times.
We first prove \cref{lem:path-order} via a simple combinatorial argument.

\begin{proof}[Proof of \cref{lem:path-order}]
    Consider the (not necessarily simple) graph $G_1$ with vertex set $S$ in which edges are drawn between $s$ and $\sigma(s)$ for each $s$. 
    We wish to draw additional edges on the vertex set $S$ in such a way that the resulting edges form in total an Eulerian cycle, and vertices have degree $4$ if they are fixed points of $\sigma$ and $2$ otherwise.
    We draw such edges (called \defn{arcs}) around the circle $C_n$ in either direction.
    Our aim is to draw the arcs in such a way that no vertex is covered by more than $12$ arcs.
    The \emph{length} of an arc is the number of edges of $C_n$ over which it lies.
    It will be helpful to direct our Eulerian cycle.

    Consider the selection of arcs which minimizes the total length of all arcs.
    \begin{claim}
        There are no three nested arcs, i.e., 
        there are no six vertices $a_1,a_2,a_3,b_3,b_2,b_1$ appearing in clockwise order such that $a_1\frown b_1$, $a_2\frown b_2$, and $a_3\frown b_3$ are all clockwise arcs.
    \end{claim}
    \begin{proofofclaim}
        Suppose such a configuration exists.
        Then, without loss of generality, there are two nested arcs $a_1\frown b_1$ and $a_2\frown b_2$ which are traversed in the same direction.
        Say without loss of generality that the direction is clockwise, and that $a_1,a_2,b_2,b_1$ appear in clockwise order.
        The Eulerian cycle is then traversed as
        \[a_1\to b_1\longrightarrow a_2\to b_2 \longrightarrow a_1,\]
        where a long arrow indicates potentially many edges of $G_1$ concatenated with many arcs. 
        The arcs $a_1\frown b_1$ and $a_2\frown b_2$ thus may be replaced by the shorter-in-sum arcs $a_1\frown a_2$ and $b_1\frown b_2$; the resulting graph will still form an Eulerian cycle.
        This contradicts our assumption on the configuration of arcs.
    \end{proofofclaim}

    \begin{claim}
        There are no three arcs $a_1\frown b_1$, $a_2\frown b_2$, and $a_3\frown b_3$ with endpoints appearing in the order $a_1,a_2,a_3,b_1,b_2,b_3$.
    \end{claim}
    \begin{proofofclaim}
        Suppose such a configuration exists.
        Then there are two arcs $a_1\frown b_1$ and $a_2\frown b_2$ which are traversed in the same direction and with endpoint order $a_1,a_2,b_1,b_2$.
        
        The Eulerian cycle is then traversed as
        \[a_1\to b_1\longrightarrow a_2\to b_2\longrightarrow a_1.\]
        The arcs $a_1\frown b_1$ and $a_2\frown b_2$ thus may be replaced by the shorter-in-sum arcs $a_1\frown a_2$ and $b_1\frown b_2$; the resulting graph will still form an Eulerian cycle.
        This contradicts our assumption on the configuration of arcs.
    \end{proofofclaim}

    Now, suppose that some vertex is covered by some arcs $a_i\frown b_i$ for $1\leq i\leq m$, where the arcs are oriented clockwise from $a_i$ to $b_i$ and where $a_1,\ldots,a_m$ appear clockwise in this order. 
    At the cost of a factor of at most three (due to fixed points of $\sigma$) we may assume that the~$a_i$ and the $b_i$ are all pairwise distinct.
    Consider the permutation $\tau$ such that the $b_i$ appear clockwise in the order $b_{\tau(1)},\ldots,b_{\tau(m)}$. By our two claims, this permutation has neither the patterns $123$ nor $321$. The Erd\H{o}s--Szekeres theorem thus implies that $m\leq 4$.
		Consequently, no vertex is covered by more than twelve arcs.
\end{proof} 

We now prove \cref{lem:path-connect}, which allows us to connect, under fairly mild conditions, many pairs of vertices by internally vertex-disjoint paths from a distributed set $S$. 
The proof proceeds in three steps:

\begin{enumerate}
    \item For each pair $(v_i,w_i)$ to be connected, ``interpolate'' the arc $v_i\frown w_i$ by vertices of $V(C_n)$ at distance at most $\alpha k/6$.

    \item For each resulting vertex (as well as each $v_i$ and $w_i$), find a nearby vertex of $S$.

    \item Using \cref{lem:nbhd-inter}, connect the chosen nearby vertices of $S$ along each arc, as well as the endpoints $v_i$ and $w_i$, by short paths with internal vertices chosen from $S$. 
\end{enumerate}

It is the third step where we use the property that distributed sets have appropriately controlled intersections with $N_G(x)\cap N_G(y)$ for vertices $x$ and $y$ of $G$. We begin by writing this step in slightly more generality.

\begin{lemma}[Pair connection]\label{lem:pair-connect}
    Let $S$ be an $\eps$-distributed set, and let $T$ be an $(\eps',r)$-locally sparse set.
    Let $H$ be a graph with vertex set $T$ and maximum degree at most $\Delta$ for which, if~$xy\in E(H)$, then $\dist_C(x,y)\leq \alpha k$.
    Suppose that $rk^{-1/5}<\eps'<\frac{\alpha\eps}{12\Delta}$.
    
    Then one can find an injection $\phi\colon E(H)\to S$ such that, for each $e\in E(H)$, the vertex $\phi(e)\in S$ is adjacent in $G$ to both endpoints of $e$.
\end{lemma}

\begin{proof}
    We define $\phi$ greedily.
    Order the edges of $H$ arbitrarily.
    Suppose we wish to select $\phi(xy)$ for some $xy\in E(H)$, and let $S'$ be the set of vertices of $S$ already mapped to by $\phi$. 
    Since $S$ is $\eps$-distributed and $\dist_C(x,y)\leq\alpha k$, we have by \cref{lem:nbhd-inter} that
    \[\abs{S\cap N_G(x)\cap N_G(y)}\geq \eps\abs{N_G(x)\cap N_G(y)}-k^{4/5}\geq\eps(\alpha k-1)-k^{4/5}\geq\frac{\alpha\eps}2k.\]
    On the other hand, consider $S'\cap N_G(x)\cap N_G(y)$. 
    Each vertex $v$ in this set can be written as~$\phi(x'y')$ for some $x'y'\in E(H)$; we have
    \[\dist_C(x,x')\leq\dist_C(x,v)+\dist_C(v,x')\leq 2k.\]
    Since $T$ is $(\eps',r)$-locally sparse, there are at most $4\eps'k+2rk^{4/5}$ vertices $x'$ satisfying the above, and thus at most $(4\eps'k+2rk^{4/5})\Delta\leq 6\Delta\eps'k$ edges $x'y'$ under consideration. 
    We conclude that
    \[\abs{S'\cap N_G(x)\cap N_G(y)}\leq 6\Delta\eps'k<\frac{\alpha\eps}2k\leq\abs{S\cap N_G(x)\cap N_G(y)}.\]
    As a result, at each step one may choose $\phi(xy)\in S\setminus S'$. This concludes the proof.
\end{proof}

We may now prove our main path connection lemma.

\begin{proof}[Proof of \cref{lem:path-connect}]

    First, partition $S$ into a $16\eps'$-distributed set $S_1$ and an $(\eps-16\eps')$-distributed set $S_2$ using \cref{lem:distributed-partition}.

    Now, within each interval $I_i$, pick a subset $J_i\subset I_i$ containing the endpoints of $I_i$ such that, when the elements of $J_i$ are enumerated in clockwise order, consecutive pairs of vertices are within $C_n$-distance at least $\alpha k/6$ and at most $\alpha k/2$. Let $\mathcal J=\bigsqcup_{i=1}^tJ_i$ be a multiset.

    \begin{claim}
        The multiset $\mathcal J$ is $2\eps'$-locally sparse.
    \end{claim}
    
    \begin{proofofclaim}
        Write $d:=\alpha k/6$ for notational simplicity. 
        Fix some interval $I$ of length at most $2k+1$.
        Since $T$ is $\eps'$-locally sparse, we have $\abs{I\cap T}\leq\eps'\abs{I}+k^{4/5}$. 
        Consider for each $i$ the intersection $(J_i\cap I)\setminus T$. 
        If the interval $I_i$ intersects $I$ in an interval of length less than~$d$, then this intersection is empty, as no element of $J_i\setminus\{v_i,w_i\}$ is within $d$ of $v_i$ or $w_i$.
        If $(J_i\cap I)\setminus T$ is nonempty, it has size at most $\abs{I\cap I_i}/d+1$.
        Now, since each point of $V(C_n)$ is in at most~$M$ intervals $I_i$, the total length of all such intervals is at most $M\cdot\abs{I}$. In particular, there are at most $M\cdot\abs{I}/d$ such intervals $I_i$. We conclude that
        \[\abs{(\mathcal J\cap I)\setminus T}\leq\frac1d\cdot M\cdot\abs{I}+\frac1d\cdot M\cdot\abs{I}=2M\cdot\frac{\abs{I}}d.\]
        As a result,
        \[\abs{\mathcal J\cap I}\leq\left(\eps'+2Md^{-1}\right)\abs{I}+k^{4/5}.\]
        The claim follows from the inequality $\eps'>12M\alpha^{-1}k^{-1}$.
    \end{proofofclaim}

    We now pair elements of $\mathcal J$ with nearby elements of $S_1$. 
    Construct a bipartite graph $H$ on vertex set $\mathcal J\sqcup S_1$ in which vertices at distance at most $\alpha k/4$ are connected by an edge. 
    Since~$S_1$ is $16\eps'$-locally dense, each element of $\mathcal J$ has degree at least $(4\alpha\eps')k-k^{4/5}$ to~$S_1$. 
    On the other hand, since $\mathcal J$ is $2\eps'$-locally sparse, each element of $S_1$ has degree at most $(2\alpha\eps')k+k^{4/5}$ to~$\mathcal J$.\footnote{In both of these statements, we have accepted losses of multiplicative constants to deal with rounding errors in the lengths of the corresponding intervals.} 
    We compute
    \[(2\alpha\eps')k+k^{4/5}<(3\alpha\eps')k<(4\alpha\eps')k-k^{4/5}.\]
    Therefore, by Hall's theorem, there is an injection $\sigma\colon\mathcal J\to S_1$ for which $\dist_C(x,\sigma(x))\leq\alpha k$ for all $x\in \mathcal J$.
    For each $1\leq i\leq t$, if $J_i$ is enumerated clockwise
    \[
        v_i=x_{i,0},x_{i,1},\ldots,x_{i,\ell_i'-1},x_{i,\ell_i}=w_i,
    \] 
    let $P_i'$ be the sequence of vertices of $V(G)$ given by
    \[P_i':=\big(v_i,\sigma(v_i),\sigma(x_{i,1}),\ldots,\sigma(x_{i,\ell_i'-1}),\sigma(w_i),w_i\big),\]
    so that the internal vertices of $P_i'$ lie in $S_1$ and, by the triangle inequality, consecutive pairs of vertices in $P_i'$ are within distance at most $\alpha k$.

    Our final step is to connect consecutive pairs of vertices in the $P_i'$ by elements of $S_2$. Since the set $\mathcal P'=\bigcup_iP_i'$ is contained in $T\sqcup S_1$, it is $(17\eps',2)$-locally sparse. So, the result follows from \cref{lem:pair-connect}. 
\end{proof}

\section{Construction of the absorber}\label{sec:absorber-constr}

The main object of this section is to prove \cref{lem:absorber-constr}.
The following gadget, which serves as a single component of our absorber, has appeared in earlier work. 

\begin{definition}\label{def:beehive}
    An \defn{$(s, x, y)$-beehive} is a subgraph $O$ of $G$ consisting of:
    \begin{itemize}
        \item two vertex-disjoint paths $P = x\dd p_1 \dd \hdots \dd p_\ell \dd y$ and $Q = s \dd q_1 \dd \hdots \dd q_\ell \dd y$ of the same length from $x$ to $y$ and $s$ to $y$, respectively;
        \item an edge $sx$; and
        \item a collection of $\ell$ vertex-disjoint paths $(U_i)_{i=1}^\ell$ from $p_i$ to $q_i$ such that each path $U_i$ intersects $P \cup Q$ only in its ends.
    \end{itemize}
    We call $\ell$ the \defn{height} of $O$.
\end{definition}

\noindent A chain of beehive absorbers is drawn in \cref{fig:beehive-line}, which we have borrowed from Bedert, Draganić, Müyesser and Pavez-Signé~\cite{bedert2026lov}.

\begin{figure}[!ht]
    \centering

\tikzset{every picture/.style={line width=0.75pt}} 

\begin{tikzpicture}[x=0.75pt,y=0.75pt,yscale=-1,xscale=1]

\draw  [dash pattern={on 0.84pt off 2.51pt}]  (174.19,108.27) -- (174.17,89.26) ;
\draw [shift={(174.17,89.26)}, rotate = 269.92] [color={rgb, 255:red, 0; green, 0; blue, 0 }  ][fill={rgb, 255:red, 0; green, 0; blue, 0 }  ][line width=0.75]      (0, 0) circle [x radius= 3.35, y radius= 3.35]   ;
\draw [shift={(174.19,108.27)}, rotate = 269.92] [color={rgb, 255:red, 0; green, 0; blue, 0 }  ][fill={rgb, 255:red, 0; green, 0; blue, 0 }  ][line width=0.75]      (0, 0) circle [x radius= 3.35, y radius= 3.35]   ;
\draw  [dash pattern={on 0.84pt off 2.51pt}]  (126.98,108.33) -- (126.96,89.7) ;
\draw [shift={(126.96,89.7)}, rotate = 269.92] [color={rgb, 255:red, 0; green, 0; blue, 0 }  ][fill={rgb, 255:red, 0; green, 0; blue, 0 }  ][line width=0.75]      (0, 0) circle [x radius= 3.35, y radius= 3.35]   ;
\draw [shift={(126.98,108.33)}, rotate = 269.92] [color={rgb, 255:red, 0; green, 0; blue, 0 }  ][fill={rgb, 255:red, 0; green, 0; blue, 0 }  ][line width=0.75]      (0, 0) circle [x radius= 3.35, y radius= 3.35]   ;
\draw    (150.01,119.48) -- (174.19,108.27) ;
\draw [shift={(174.19,108.27)}, rotate = 335.12] [color={rgb, 255:red, 0; green, 0; blue, 0 }  ][fill={rgb, 255:red, 0; green, 0; blue, 0 }  ][line width=0.75]      (0, 0) circle [x radius= 3.35, y radius= 3.35]   ;
\draw [shift={(150.01,119.48)}, rotate = 335.12] [color={rgb, 255:red, 0; green, 0; blue, 0 }  ][fill={rgb, 255:red, 0; green, 0; blue, 0 }  ][line width=0.75]      (0, 0) circle [x radius= 3.35, y radius= 3.35]   ;
\draw    (150.01,119.48) -- (126.98,108.33) ;
\draw [shift={(126.98,108.33)}, rotate = 205.83] [color={rgb, 255:red, 0; green, 0; blue, 0 }  ][fill={rgb, 255:red, 0; green, 0; blue, 0 }  ][line width=0.75]      (0, 0) circle [x radius= 3.35, y radius= 3.35]   ;
\draw [shift={(150.01,119.48)}, rotate = 205.83] [color={rgb, 255:red, 0; green, 0; blue, 0 }  ][fill={rgb, 255:red, 0; green, 0; blue, 0 }  ][line width=0.75]      (0, 0) circle [x radius= 3.35, y radius= 3.35]   ;
\draw  [dash pattern={on 0.84pt off 2.51pt}]  (162.34,71.39) -- (138.14,71.42) ;
\draw [shift={(138.14,71.42)}, rotate = 179.92] [color={rgb, 255:red, 0; green, 0; blue, 0 }  ][fill={rgb, 255:red, 0; green, 0; blue, 0 }  ][line width=0.75]      (0, 0) circle [x radius= 3.35, y radius= 3.35]   ;
\draw [shift={(162.34,71.39)}, rotate = 179.92] [color={rgb, 255:red, 0; green, 0; blue, 0 }  ][fill={rgb, 255:red, 0; green, 0; blue, 0 }  ][line width=0.75]      (0, 0) circle [x radius= 3.35, y radius= 3.35]   ;
\draw    (174.19,108.27) .. controls (172.03,109.21) and (170.48,108.6) .. (169.54,106.44) .. controls (168.6,104.28) and (167.05,103.67) .. (164.89,104.61) .. controls (162.73,105.55) and (161.18,104.94) .. (160.23,102.78) .. controls (159.29,100.62) and (157.74,100.01) .. (155.58,100.95) .. controls (153.42,101.89) and (151.87,101.28) .. (150.93,99.12) .. controls (149.98,96.96) and (148.43,96.35) .. (146.27,97.29) .. controls (144.11,98.23) and (142.56,97.62) .. (141.62,95.46) .. controls (140.68,93.3) and (139.13,92.69) .. (136.97,93.63) .. controls (134.81,94.58) and (133.26,93.97) .. (132.31,91.81) .. controls (131.37,89.65) and (129.82,89.04) .. (127.66,89.98) -- (126.96,89.7) -- (126.96,89.7) ;
\draw    (126.96,89.7) -- (138.14,71.42) ;
\draw [shift={(138.14,71.42)}, rotate = 301.47] [color={rgb, 255:red, 0; green, 0; blue, 0 }  ][fill={rgb, 255:red, 0; green, 0; blue, 0 }  ][line width=0.75]      (0, 0) circle [x radius= 3.35, y radius= 3.35]   ;
\draw [shift={(126.96,89.7)}, rotate = 301.47] [color={rgb, 255:red, 0; green, 0; blue, 0 }  ][fill={rgb, 255:red, 0; green, 0; blue, 0 }  ][line width=0.75]      (0, 0) circle [x radius= 3.35, y radius= 3.35]   ;
\draw    (174.17,89.26) -- (162.34,71.39) ;
\draw [shift={(162.34,71.39)}, rotate = 236.5] [color={rgb, 255:red, 0; green, 0; blue, 0 }  ][fill={rgb, 255:red, 0; green, 0; blue, 0 }  ][line width=0.75]      (0, 0) circle [x radius= 3.35, y radius= 3.35]   ;
\draw [shift={(174.17,89.26)}, rotate = 236.5] [color={rgb, 255:red, 0; green, 0; blue, 0 }  ][fill={rgb, 255:red, 0; green, 0; blue, 0 }  ][line width=0.75]      (0, 0) circle [x radius= 3.35, y radius= 3.35]   ;
\draw    (174.17,89.26) .. controls (171.93,90.01) and (170.44,89.26) .. (169.69,87.03) .. controls (168.95,84.79) and (167.46,84.04) .. (165.22,84.79) .. controls (162.99,85.54) and (161.5,84.79) .. (160.75,82.56) .. controls (160,80.33) and (158.5,79.58) .. (156.27,80.33) .. controls (154.03,81.08) and (152.54,80.33) .. (151.8,78.09) .. controls (151.05,75.86) and (149.56,75.11) .. (147.33,75.86) .. controls (145.1,76.61) and (143.6,75.86) .. (142.85,73.63) -- (138.44,71.42) -- (138.44,71.42) ;
\draw    (286.19,108.27) -- (286.17,89.26) ;
\draw [shift={(286.17,89.26)}, rotate = 269.92] [color={rgb, 255:red, 0; green, 0; blue, 0 }  ][fill={rgb, 255:red, 0; green, 0; blue, 0 }  ][line width=0.75]      (0, 0) circle [x radius= 3.35, y radius= 3.35]   ;
\draw [shift={(286.19,108.27)}, rotate = 269.92] [color={rgb, 255:red, 0; green, 0; blue, 0 }  ][fill={rgb, 255:red, 0; green, 0; blue, 0 }  ][line width=0.75]      (0, 0) circle [x radius= 3.35, y radius= 3.35]   ;
\draw    (238.98,108.33) -- (238.96,89.7) ;
\draw [shift={(238.96,89.7)}, rotate = 269.92] [color={rgb, 255:red, 0; green, 0; blue, 0 }  ][fill={rgb, 255:red, 0; green, 0; blue, 0 }  ][line width=0.75]      (0, 0) circle [x radius= 3.35, y radius= 3.35]   ;
\draw [shift={(238.98,108.33)}, rotate = 269.92] [color={rgb, 255:red, 0; green, 0; blue, 0 }  ][fill={rgb, 255:red, 0; green, 0; blue, 0 }  ][line width=0.75]      (0, 0) circle [x radius= 3.35, y radius= 3.35]   ;
\draw  [dash pattern={on 0.84pt off 2.51pt}]  (262.01,119.48) -- (286.19,108.27) ;
\draw [shift={(286.19,108.27)}, rotate = 335.12] [color={rgb, 255:red, 0; green, 0; blue, 0 }  ][fill={rgb, 255:red, 0; green, 0; blue, 0 }  ][line width=0.75]      (0, 0) circle [x radius= 3.35, y radius= 3.35]   ;
\draw [shift={(262.01,119.48)}, rotate = 335.12] [color={rgb, 255:red, 0; green, 0; blue, 0 }  ][fill={rgb, 255:red, 0; green, 0; blue, 0 }  ][line width=0.75]      (0, 0) circle [x radius= 3.35, y radius= 3.35]   ;
\draw  [dash pattern={on 0.84pt off 2.51pt}]  (262.01,119.48) -- (238.98,108.33) ;
\draw [shift={(238.98,108.33)}, rotate = 205.83] [color={rgb, 255:red, 0; green, 0; blue, 0 }  ][fill={rgb, 255:red, 0; green, 0; blue, 0 }  ][line width=0.75]      (0, 0) circle [x radius= 3.35, y radius= 3.35]   ;
\draw [shift={(262.01,119.48)}, rotate = 205.83] [color={rgb, 255:red, 0; green, 0; blue, 0 }  ][fill={rgb, 255:red, 0; green, 0; blue, 0 }  ][line width=0.75]      (0, 0) circle [x radius= 3.35, y radius= 3.35]   ;
\draw    (274.34,71.39) -- (250.14,71.42) ;
\draw [shift={(250.14,71.42)}, rotate = 179.92] [color={rgb, 255:red, 0; green, 0; blue, 0 }  ][fill={rgb, 255:red, 0; green, 0; blue, 0 }  ][line width=0.75]      (0, 0) circle [x radius= 3.35, y radius= 3.35]   ;
\draw [shift={(274.34,71.39)}, rotate = 179.92] [color={rgb, 255:red, 0; green, 0; blue, 0 }  ][fill={rgb, 255:red, 0; green, 0; blue, 0 }  ][line width=0.75]      (0, 0) circle [x radius= 3.35, y radius= 3.35]   ;
\draw    (286.19,108.27) .. controls (284.03,109.21) and (282.48,108.6) .. (281.54,106.44) .. controls (280.6,104.28) and (279.05,103.67) .. (276.89,104.61) .. controls (274.73,105.55) and (273.18,104.94) .. (272.23,102.78) .. controls (271.29,100.62) and (269.74,100.01) .. (267.58,100.95) .. controls (265.42,101.89) and (263.87,101.28) .. (262.93,99.12) .. controls (261.98,96.96) and (260.43,96.35) .. (258.27,97.29) .. controls (256.11,98.23) and (254.56,97.62) .. (253.62,95.46) .. controls (252.68,93.3) and (251.13,92.69) .. (248.97,93.63) .. controls (246.81,94.58) and (245.26,93.97) .. (244.31,91.81) .. controls (243.37,89.65) and (241.82,89.04) .. (239.66,89.98) -- (238.96,89.7) -- (238.96,89.7) ;
\draw  [dash pattern={on 0.84pt off 2.51pt}]  (238.96,89.7) -- (250.14,71.42) ;
\draw [shift={(250.14,71.42)}, rotate = 301.47] [color={rgb, 255:red, 0; green, 0; blue, 0 }  ][fill={rgb, 255:red, 0; green, 0; blue, 0 }  ][line width=0.75]      (0, 0) circle [x radius= 3.35, y radius= 3.35]   ;
\draw [shift={(238.96,89.7)}, rotate = 301.47] [color={rgb, 255:red, 0; green, 0; blue, 0 }  ][fill={rgb, 255:red, 0; green, 0; blue, 0 }  ][line width=0.75]      (0, 0) circle [x radius= 3.35, y radius= 3.35]   ;
\draw  [dash pattern={on 0.84pt off 2.51pt}]  (286.17,89.26) -- (274.34,71.39) ;
\draw [shift={(274.34,71.39)}, rotate = 236.5] [color={rgb, 255:red, 0; green, 0; blue, 0 }  ][fill={rgb, 255:red, 0; green, 0; blue, 0 }  ][line width=0.75]      (0, 0) circle [x radius= 3.35, y radius= 3.35]   ;
\draw [shift={(286.17,89.26)}, rotate = 236.5] [color={rgb, 255:red, 0; green, 0; blue, 0 }  ][fill={rgb, 255:red, 0; green, 0; blue, 0 }  ][line width=0.75]      (0, 0) circle [x radius= 3.35, y radius= 3.35]   ;
\draw    (286.17,89.26) .. controls (283.93,90.01) and (282.44,89.26) .. (281.69,87.03) .. controls (280.95,84.79) and (279.46,84.04) .. (277.22,84.79) .. controls (274.99,85.54) and (273.5,84.79) .. (272.75,82.56) .. controls (272,80.33) and (270.5,79.58) .. (268.27,80.33) .. controls (266.03,81.08) and (264.54,80.33) .. (263.8,78.09) .. controls (263.05,75.86) and (261.56,75.11) .. (259.33,75.86) .. controls (257.1,76.61) and (255.6,75.86) .. (254.85,73.63) -- (250.44,71.42) -- (250.44,71.42) ;
\draw    (398.19,107.27) -- (398.17,88.26) ;
\draw [shift={(398.17,88.26)}, rotate = 269.92] [color={rgb, 255:red, 0; green, 0; blue, 0 }  ][fill={rgb, 255:red, 0; green, 0; blue, 0 }  ][line width=0.75]      (0, 0) circle [x radius= 3.35, y radius= 3.35]   ;
\draw [shift={(398.19,107.27)}, rotate = 269.92] [color={rgb, 255:red, 0; green, 0; blue, 0 }  ][fill={rgb, 255:red, 0; green, 0; blue, 0 }  ][line width=0.75]      (0, 0) circle [x radius= 3.35, y radius= 3.35]   ;
\draw    (350.98,107.33) -- (350.96,88.7) ;
\draw [shift={(350.96,88.7)}, rotate = 269.92] [color={rgb, 255:red, 0; green, 0; blue, 0 }  ][fill={rgb, 255:red, 0; green, 0; blue, 0 }  ][line width=0.75]      (0, 0) circle [x radius= 3.35, y radius= 3.35]   ;
\draw [shift={(350.98,107.33)}, rotate = 269.92] [color={rgb, 255:red, 0; green, 0; blue, 0 }  ][fill={rgb, 255:red, 0; green, 0; blue, 0 }  ][line width=0.75]      (0, 0) circle [x radius= 3.35, y radius= 3.35]   ;
\draw  [dash pattern={on 0.84pt off 2.51pt}]  (374.01,118.48) -- (398.19,107.27) ;
\draw [shift={(398.19,107.27)}, rotate = 335.12] [color={rgb, 255:red, 0; green, 0; blue, 0 }  ][fill={rgb, 255:red, 0; green, 0; blue, 0 }  ][line width=0.75]      (0, 0) circle [x radius= 3.35, y radius= 3.35]   ;
\draw [shift={(374.01,118.48)}, rotate = 335.12] [color={rgb, 255:red, 0; green, 0; blue, 0 }  ][fill={rgb, 255:red, 0; green, 0; blue, 0 }  ][line width=0.75]      (0, 0) circle [x radius= 3.35, y radius= 3.35]   ;
\draw  [dash pattern={on 0.84pt off 2.51pt}]  (374.01,118.48) -- (350.98,107.33) ;
\draw [shift={(350.98,107.33)}, rotate = 205.83] [color={rgb, 255:red, 0; green, 0; blue, 0 }  ][fill={rgb, 255:red, 0; green, 0; blue, 0 }  ][line width=0.75]      (0, 0) circle [x radius= 3.35, y radius= 3.35]   ;
\draw [shift={(374.01,118.48)}, rotate = 205.83] [color={rgb, 255:red, 0; green, 0; blue, 0 }  ][fill={rgb, 255:red, 0; green, 0; blue, 0 }  ][line width=0.75]      (0, 0) circle [x radius= 3.35, y radius= 3.35]   ;
\draw    (386.34,70.39) -- (362.14,70.42) ;
\draw [shift={(362.14,70.42)}, rotate = 179.92] [color={rgb, 255:red, 0; green, 0; blue, 0 }  ][fill={rgb, 255:red, 0; green, 0; blue, 0 }  ][line width=0.75]      (0, 0) circle [x radius= 3.35, y radius= 3.35]   ;
\draw [shift={(386.34,70.39)}, rotate = 179.92] [color={rgb, 255:red, 0; green, 0; blue, 0 }  ][fill={rgb, 255:red, 0; green, 0; blue, 0 }  ][line width=0.75]      (0, 0) circle [x radius= 3.35, y radius= 3.35]   ;
\draw    (398.19,107.27) .. controls (396.03,108.21) and (394.48,107.6) .. (393.54,105.44) .. controls (392.6,103.28) and (391.05,102.67) .. (388.89,103.61) .. controls (386.73,104.55) and (385.18,103.94) .. (384.23,101.78) .. controls (383.29,99.62) and (381.74,99.01) .. (379.58,99.95) .. controls (377.42,100.89) and (375.87,100.28) .. (374.93,98.12) .. controls (373.98,95.96) and (372.43,95.35) .. (370.27,96.29) .. controls (368.11,97.23) and (366.56,96.62) .. (365.62,94.46) .. controls (364.68,92.3) and (363.13,91.69) .. (360.97,92.63) .. controls (358.81,93.58) and (357.26,92.97) .. (356.31,90.81) .. controls (355.37,88.65) and (353.82,88.04) .. (351.66,88.98) -- (350.96,88.7) -- (350.96,88.7) ;
\draw  [dash pattern={on 0.84pt off 2.51pt}]  (350.96,88.7) -- (362.14,70.42) ;
\draw [shift={(362.14,70.42)}, rotate = 301.47] [color={rgb, 255:red, 0; green, 0; blue, 0 }  ][fill={rgb, 255:red, 0; green, 0; blue, 0 }  ][line width=0.75]      (0, 0) circle [x radius= 3.35, y radius= 3.35]   ;
\draw [shift={(350.96,88.7)}, rotate = 301.47] [color={rgb, 255:red, 0; green, 0; blue, 0 }  ][fill={rgb, 255:red, 0; green, 0; blue, 0 }  ][line width=0.75]      (0, 0) circle [x radius= 3.35, y radius= 3.35]   ;
\draw  [dash pattern={on 0.84pt off 2.51pt}]  (398.17,88.26) -- (386.34,70.39) ;
\draw [shift={(386.34,70.39)}, rotate = 236.5] [color={rgb, 255:red, 0; green, 0; blue, 0 }  ][fill={rgb, 255:red, 0; green, 0; blue, 0 }  ][line width=0.75]      (0, 0) circle [x radius= 3.35, y radius= 3.35]   ;
\draw [shift={(398.17,88.26)}, rotate = 236.5] [color={rgb, 255:red, 0; green, 0; blue, 0 }  ][fill={rgb, 255:red, 0; green, 0; blue, 0 }  ][line width=0.75]      (0, 0) circle [x radius= 3.35, y radius= 3.35]   ;
\draw    (398.17,88.26) .. controls (395.93,89.01) and (394.44,88.26) .. (393.69,86.03) .. controls (392.95,83.79) and (391.46,83.04) .. (389.22,83.79) .. controls (386.99,84.54) and (385.5,83.79) .. (384.75,81.56) .. controls (384,79.33) and (382.5,78.58) .. (380.27,79.33) .. controls (378.03,80.08) and (376.54,79.33) .. (375.8,77.09) .. controls (375.05,74.86) and (373.56,74.11) .. (371.33,74.86) .. controls (369.1,75.61) and (367.6,74.86) .. (366.85,72.63) -- (362.44,70.42) -- (362.44,70.42) ;
\draw  [dash pattern={on 0.84pt off 2.51pt}]  (509.19,107.27) -- (509.17,88.26) ;
\draw [shift={(509.17,88.26)}, rotate = 269.92] [color={rgb, 255:red, 0; green, 0; blue, 0 }  ][fill={rgb, 255:red, 0; green, 0; blue, 0 }  ][line width=0.75]      (0, 0) circle [x radius= 3.35, y radius= 3.35]   ;
\draw [shift={(509.19,107.27)}, rotate = 269.92] [color={rgb, 255:red, 0; green, 0; blue, 0 }  ][fill={rgb, 255:red, 0; green, 0; blue, 0 }  ][line width=0.75]      (0, 0) circle [x radius= 3.35, y radius= 3.35]   ;
\draw  [dash pattern={on 0.84pt off 2.51pt}]  (461.98,107.33) -- (461.96,88.7) ;
\draw [shift={(461.96,88.7)}, rotate = 269.92] [color={rgb, 255:red, 0; green, 0; blue, 0 }  ][fill={rgb, 255:red, 0; green, 0; blue, 0 }  ][line width=0.75]      (0, 0) circle [x radius= 3.35, y radius= 3.35]   ;
\draw [shift={(461.98,107.33)}, rotate = 269.92] [color={rgb, 255:red, 0; green, 0; blue, 0 }  ][fill={rgb, 255:red, 0; green, 0; blue, 0 }  ][line width=0.75]      (0, 0) circle [x radius= 3.35, y radius= 3.35]   ;
\draw    (485.01,118.48) -- (509.19,107.27) ;
\draw [shift={(509.19,107.27)}, rotate = 335.12] [color={rgb, 255:red, 0; green, 0; blue, 0 }  ][fill={rgb, 255:red, 0; green, 0; blue, 0 }  ][line width=0.75]      (0, 0) circle [x radius= 3.35, y radius= 3.35]   ;
\draw [shift={(485.01,118.48)}, rotate = 335.12] [color={rgb, 255:red, 0; green, 0; blue, 0 }  ][fill={rgb, 255:red, 0; green, 0; blue, 0 }  ][line width=0.75]      (0, 0) circle [x radius= 3.35, y radius= 3.35]   ;
\draw    (485.01,118.48) -- (461.98,107.33) ;
\draw [shift={(461.98,107.33)}, rotate = 205.83] [color={rgb, 255:red, 0; green, 0; blue, 0 }  ][fill={rgb, 255:red, 0; green, 0; blue, 0 }  ][line width=0.75]      (0, 0) circle [x radius= 3.35, y radius= 3.35]   ;
\draw [shift={(485.01,118.48)}, rotate = 205.83] [color={rgb, 255:red, 0; green, 0; blue, 0 }  ][fill={rgb, 255:red, 0; green, 0; blue, 0 }  ][line width=0.75]      (0, 0) circle [x radius= 3.35, y radius= 3.35]   ;
\draw  [dash pattern={on 0.84pt off 2.51pt}]  (497.34,70.39) -- (473.14,70.42) ;
\draw [shift={(473.14,70.42)}, rotate = 179.92] [color={rgb, 255:red, 0; green, 0; blue, 0 }  ][fill={rgb, 255:red, 0; green, 0; blue, 0 }  ][line width=0.75]      (0, 0) circle [x radius= 3.35, y radius= 3.35]   ;
\draw [shift={(497.34,70.39)}, rotate = 179.92] [color={rgb, 255:red, 0; green, 0; blue, 0 }  ][fill={rgb, 255:red, 0; green, 0; blue, 0 }  ][line width=0.75]      (0, 0) circle [x radius= 3.35, y radius= 3.35]   ;
\draw    (509.19,107.27) .. controls (507.03,108.21) and (505.48,107.6) .. (504.54,105.44) .. controls (503.6,103.28) and (502.05,102.67) .. (499.89,103.61) .. controls (497.73,104.55) and (496.18,103.94) .. (495.23,101.78) .. controls (494.29,99.62) and (492.74,99.01) .. (490.58,99.95) .. controls (488.42,100.89) and (486.87,100.28) .. (485.93,98.12) .. controls (484.98,95.96) and (483.43,95.35) .. (481.27,96.29) .. controls (479.11,97.23) and (477.56,96.62) .. (476.62,94.46) .. controls (475.68,92.3) and (474.13,91.69) .. (471.97,92.63) .. controls (469.81,93.58) and (468.26,92.97) .. (467.31,90.81) .. controls (466.37,88.65) and (464.82,88.04) .. (462.66,88.98) -- (461.96,88.7) -- (461.96,88.7) ;
\draw    (461.96,88.7) -- (473.14,70.42) ;
\draw [shift={(473.14,70.42)}, rotate = 301.47] [color={rgb, 255:red, 0; green, 0; blue, 0 }  ][fill={rgb, 255:red, 0; green, 0; blue, 0 }  ][line width=0.75]      (0, 0) circle [x radius= 3.35, y radius= 3.35]   ;
\draw [shift={(461.96,88.7)}, rotate = 301.47] [color={rgb, 255:red, 0; green, 0; blue, 0 }  ][fill={rgb, 255:red, 0; green, 0; blue, 0 }  ][line width=0.75]      (0, 0) circle [x radius= 3.35, y radius= 3.35]   ;
\draw    (509.17,88.26) -- (497.34,70.39) ;
\draw [shift={(497.34,70.39)}, rotate = 236.5] [color={rgb, 255:red, 0; green, 0; blue, 0 }  ][fill={rgb, 255:red, 0; green, 0; blue, 0 }  ][line width=0.75]      (0, 0) circle [x radius= 3.35, y radius= 3.35]   ;
\draw [shift={(509.17,88.26)}, rotate = 236.5] [color={rgb, 255:red, 0; green, 0; blue, 0 }  ][fill={rgb, 255:red, 0; green, 0; blue, 0 }  ][line width=0.75]      (0, 0) circle [x radius= 3.35, y radius= 3.35]   ;
\draw    (509.17,88.26) .. controls (506.93,89.01) and (505.44,88.26) .. (504.69,86.03) .. controls (503.95,83.79) and (502.46,83.04) .. (500.22,83.79) .. controls (497.99,84.54) and (496.5,83.79) .. (495.75,81.56) .. controls (495,79.33) and (493.5,78.58) .. (491.27,79.33) .. controls (489.03,80.08) and (487.54,79.33) .. (486.8,77.09) .. controls (486.05,74.86) and (484.56,74.11) .. (482.33,74.86) .. controls (480.1,75.61) and (478.6,74.86) .. (477.85,72.63) -- (473.44,70.42) -- (473.44,70.42) ;
\draw    (162.34,71.39) .. controls (198,75) and (218,87) .. (238.98,108.33) ;
\draw    (274.34,71.39) .. controls (310,75) and (330,87) .. (350.98,108.33) ;
\draw    (386.34,70.39) .. controls (422,74) and (442,86) .. (462.98,107.33) ;
\draw    (126.98,108.33) .. controls (125.31,110) and (123.65,110.01) .. (121.98,108.34) .. controls (120.31,106.68) and (118.64,106.69) .. (116.98,108.36) .. controls (115.31,110.03) and (113.65,110.03) .. (111.98,108.37) .. controls (110.31,106.71) and (108.65,106.71) .. (106.98,108.38) .. controls (105.31,110.05) and (103.65,110.05) .. (101.98,108.39) .. controls (100.31,106.73) and (98.65,106.73) .. (96.98,108.4) .. controls (95.32,110.07) and (93.65,110.08) .. (91.98,108.42) .. controls (90.31,106.76) and (88.65,106.76) .. (86.98,108.43) .. controls (85.31,110.1) and (83.65,110.1) .. (81.98,108.44) .. controls (80.31,106.78) and (78.65,106.78) .. (76.98,108.45) .. controls (75.31,110.12) and (73.65,110.12) .. (71.98,108.46) .. controls (70.31,106.8) and (68.65,106.8) .. (66.98,108.47) .. controls (65.32,110.14) and (63.65,110.15) .. (61.98,108.49) .. controls (60.31,106.83) and (58.65,106.83) .. (56.98,108.5) -- (56,108.5) -- (56,108.5) ;
\draw    (568.32,70.22) .. controls (566.65,71.89) and (564.99,71.9) .. (563.32,70.23) .. controls (561.65,68.57) and (559.98,68.58) .. (558.32,70.25) .. controls (556.65,71.92) and (554.99,71.92) .. (553.32,70.26) .. controls (551.65,68.6) and (549.99,68.6) .. (548.32,70.27) .. controls (546.65,71.94) and (544.99,71.94) .. (543.32,70.28) .. controls (541.65,68.62) and (539.99,68.62) .. (538.32,70.29) .. controls (536.66,71.96) and (534.99,71.97) .. (533.32,70.31) .. controls (531.65,68.65) and (529.99,68.65) .. (528.32,70.32) .. controls (526.65,71.99) and (524.99,71.99) .. (523.32,70.33) .. controls (521.65,68.67) and (519.99,68.67) .. (518.32,70.34) .. controls (516.65,72.01) and (514.99,72.01) .. (513.32,70.35) .. controls (511.65,68.69) and (509.99,68.69) .. (508.32,70.36) .. controls (506.66,72.03) and (504.99,72.04) .. (503.32,70.38) .. controls (501.65,68.72) and (499.99,68.72) .. (498.32,70.39) -- (497.34,70.39) -- (497.34,70.39) ;

\draw (144,124.65) node [anchor=north west][inner sep=0.75pt]  [rotate=-359.39]  {$s$};
\draw (109,110.13) node [anchor=north west][inner sep=0.75pt]  [rotate=-0.98]  {$x$};
\draw (107,80.13) node [anchor=north west][inner sep=0.75pt]  [rotate=-0.98]  {$p_1$};
\draw (117.92,53.82) node [anchor=north west][inner sep=0.75pt]  [rotate=-359.14]  {$p_2$};
\draw (166.92,53.82) node [anchor=north west][inner sep=0.75pt]  [rotate=-359.14]  {$y$};
\draw (179,110.13) node [anchor=north west][inner sep=0.75pt]  [rotate=-0.98]  {$q_1$};
\draw (179,80.13) node [anchor=north west][inner sep=0.75pt]  [rotate=-0.98]  {$q_2$};
\draw (250.58,124.65) node [anchor=north west][inner sep=0.75pt]  [rotate=-359.39]  {};
\draw (207.64,107.13) node [anchor=north west][inner sep=0.75pt]  [rotate=-0.98]  {};
\draw (278.92,51.82) node [anchor=north west][inner sep=0.75pt]  [rotate=-359.14]  {};
\draw (362.58,123.65) node [anchor=north west][inner sep=0.75pt]  [rotate=-359.39]  {};
\draw (319.64,107.13) node [anchor=north west][inner sep=0.75pt]  [rotate=-0.98]  {};
\draw (390.92,51.82) node [anchor=north west][inner sep=0.75pt]  [rotate=-359.14]  {};
\draw (473.58,123.65) node [anchor=north west][inner sep=0.75pt]  [rotate=-359.39]  {};
\draw (431.64,107.13) node [anchor=north west][inner sep=0.75pt]  [rotate=-0.98]  {};
\draw (498.92,51.82) node [anchor=north west][inner sep=0.75pt]  [rotate=-359.14] {};
\end{tikzpicture}
    \caption{A chain of beehive absorbers of height $\ell=2$. Within each beehive, a path from the $x$-vertex (on the bottom left) to the $y$-vertex (on the top right) is drawn. The vertices of the leftmost beehive are labeled in accordance with \cref{def:beehive}.}
    \label{fig:beehive-line}
\end{figure}
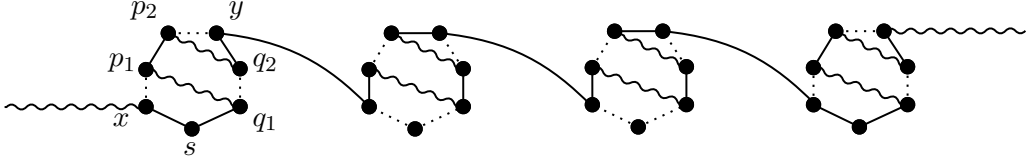

\begin{lemma}[Beehives are absorbers]\label{lem:beehive-useful}
    Every $(s,x,y)$-beehive is an $(x,y;\{s\})$-absorber.
\end{lemma}
\begin{proof}
    Let $O$ be an $(s,x,y)$-beehive; we must show that both $O$ and $O[V(O)\setminus\{s\}]$ contain spanning paths with endpoints $x$ and $y$.
    These may both be explicitly constructed.
    For $O$, take the path
    \[x\to s\to q_1\xrightarrow{U_1}p_1\to p_2\xrightarrow{U_2}q_2\to\cdots\to (p_\ell\text{ or }q_\ell)\xrightarrow{U_\ell}(q_\ell\text{ or }p_\ell)\to y,\]
    while for $O[V(O)\setminus\{s\}]$, take the path
    \[x\to p_1\xrightarrow{U_1}q_1\to q_2\xrightarrow{U_2}p_2\to\cdots\to (q_\ell\text{ or }p_\ell)\xrightarrow{U_\ell}(p_\ell\text{ or }q_\ell)\to y.\qedhere\]
\end{proof}

Recall that we wish to find an $(x_0,y_0;S)$-absorber in $R$ for some $x_0,y_0\not\in R\cup S$. We begin by finding vertex-disjoint beehives for each $s\in S$.

\begin{lemma}[Beehive construction]\label{lem:beehive-constr}
    Let $R\subset V(C_n)$ be $\eps$-distributed and let $S\subset V(C_n)\setminus R$ be $\eps'$-locally sparse.
    Suppose that $\alpha^{-1}k^{-1/5}<\eps'<10^{-14}\alpha^5\eps$.

    Then there exist, for each $s \in S$, some vertices $x^{(s)},y^{(s)}\in R$ and some $(s,x^{(s)},y^{(s)})$-beehive~$O^{(s)}$ with vertex set in $R\sqcup\{s\}$ and height at most $2+6\alpha^{-1}$. 
    Moreover, we can ensure that the vertex sets of $O^{(s)}$ are disjoint.    
\end{lemma}

We partition the distributed set $R$ into three distributed sets $R_1\sqcup R_2\sqcup R_3$. In the notation of \cref{def:beehive}, we will select the $x$-vertices of the beehives from $R_1$, the $\{p_i,q_i,y\}$-vertices from $R_2$, and the paths $U_i$ (as well as paths connecting the $x$- and $y$-vertices of adjacent beehives) from $R_3$. Our main tool will be \cref{cor:path-connect}.

\begin{proof}[Proof of \cref{lem:beehive-constr}]
    Select $\eps_1:=6\eps'$ and $\eps_2:=10^6\alpha^{-2}\eps'$, and set $\eps_3:=\eps-\eps_1-\eps_2$. 
    By \cref{lem:distributed-partition}, we can partition $R$ into three sets $R_1\sqcup R_2\sqcup R_3$ such that each $R_j$ is $\eps_j$-distributed.

    Since $R_1$ is $\eps_1$-distributed, we have
    \[\abs{R_1\cap N_G(s)}\geq \eps_1\abs{N_G(s)}-k^{4/5}\geq \eps_1k-k^{4/5}>5\eps'k\]
    for each $s\in S$. On the other hand, since $S$ is $\eps'$-locally sparse, each vertex of $R_1$ is adjacent to at most $4\eps'k+k^{4/5}<5\eps'k$ elements of $S$. So, by Hall's theorem, we can find an injection $\pi\colon S\to R_1$ such that $s\pi(s)\in E(G)$ for each $s\in S$. Write $x^{(s)}:=\pi(s)$.

    The set $S\sqcup\pi(S)$ is contained in $S\sqcup R_1$ and is thus $(\eps'+\eps_1)$-locally sparse. 
    Moreover, we have
    \[\alpha^{-1}k^{-1/5}<\eps'+\eps_1=7\eps'<10^{-5}\alpha^2\eps_2.\]
    We can thus apply \cref{cor:path-connect} with $L=1$ and the set $R_2$ to find a collection of disjoint paths $P^{(s)}$ of length $2\ell^{(s)}+2\leq 6+12\alpha^{-1}$ with internal vertices in $R_2$ which connect $s$ to $x^{(s)}$ for each $s\in S$. 
    For each $s$, enumerate the so-found path as
    \[x^{(s)},p_1^{(s)},\ldots,p_{\ell^{(s)}}^{(s)},y^{(s)},q_{\ell^{(s)}}^{(s)},\ldots,q_1^{(s)},s.\]
    
    Let
    \[T=\bigcup_{s\in S}\left\{p_1^{(s)},\ldots,p_{\ell^{(s)}}^{(s)},q_1^{(s)},\ldots,q_{\ell^{(s)}}^{(s)}\right\}\]
    with the pairing $p_j^{(s)}\leftrightarrow q_j^{(s)}$. 
    The set $T$ is contained in $S\sqcup R_1\sqcup R_2$, and is thus $(\eps'+\eps_1+\eps_2)$-locally sparse.
    Moreover, each pair $v\leftrightarrow w$ paired in $T$ is at graph distance at most $6+12\alpha^{-1}$ and thus at $C$-distance at most $(6+12\alpha^{-1})k$. 
    We conclude from \cref{cor:path-connect} applied with $L=6+12\alpha^{-1}$ and the set $R_3$ that we can find a collection of disjoint paths with internal vertices in $R_3$ connecting $p_j^{(s)}$ to $q_j^{(s)}$ for each $s$ and each $j$.

    We have now found, for each $s$, a $(s,x^{(s)},y^{(s)})$-beehive $O^{(s)}$ with vertex set (besides $s$) in~$R=R_1\sqcup R_2\sqcup R_3$. The proof is complete.
\end{proof}

\begin{proof}[Proof of \cref{lem:absorber-constr}]
    Recall that $R$ is an $\eps$-distributed set.
    Let $\eps_1=10^{14}\alpha^{-5}\eps'$ and $\eps_2=\eps-\eps_1$. 
    By \cref{lem:distributed-partition}, we can partition $R$ into two sets $R_1\sqcup R_2$ such that each $R_j$ is $\eps_j$-distributed. Apply \cref{lem:beehive-constr} to $S$ and $R_1$ to find, for each $s\in S$, vertices $x^{(s)},y^{(s)}\in R_1$ and vertex-disjoint beehives $O^{(s)}$ of height at most $\ell:=2+6\alpha^{-1}$ with vertex sets contained in $S\sqcup R_1$.
    By \cref{lem:beehive-useful}, it is enough to find some order $s_1,\ldots,s_{\abs{S}}$ of $S$ and to find vertex-disjoint paths connecting $y^{(s_i)}$ to $x^{(s_{i+1})}$ for each $1\leq i<\abs{S}$ with internal vertices lying in $R_2$.
    If we can do this, then we can form a path
    \[x^{(s_1)}\xrightarrow{O^{(s_1)}}y^{(s_1)}\to x^{(s_2)}\xrightarrow{O^{(s_2)}}y^{(s_2)}\to\cdots\to x^{(s_{\abs{S}})}\xrightarrow{O^{(s_{\abs{S}})}}y^{(s_{\abs{S}})}\]
    for each subset $S'\subset S$ (for each $s\in S$, the path chosen through $O^{(s)}$ depends on whether $s\in S'$) which witnesses the statement that $R$ contains an $(x^{(s_1)},y^{(s_{\abs{S}})};S)$-absorber.

    Order $s_1,\ldots,s_{\abs{S}}$ clockwise around $C_n$. Since $S$ is distributed, $S$ intersects each interval of length $k$. Therefore, for each $i$,
    \[\dist_C\big(y^{(s_i)},x^{(s_{i+1})}\big)\leq \dist_C\big(y^{(s_i)},s_i\big)+\dist_C(s_i,s_{i+1})+\dist_C\big(s_{i+1},x^{(s_{i+1})}\big)\leq k(\ell+3).\]
    Note also that the set of $x^{(s)}$ and $y^{(s)}$, as a subset of $S\sqcup R_1$, is $(\eps'+\eps_1)$-locally sparse. Since
    \[\alpha^{-1}k^{-1/5}<\eps'+\eps_1<10^{-7}\alpha^3\eps<10^{-5}(\ell+3)^{-1}\alpha^2\eps_2,\]
    we may thus apply \cref{cor:path-connect} to find our desired paths in $R_2$. 
\end{proof}

\section{Hall's condition}\label{sec:Hall}

We say that a graph $G$ \defn{satisfies Hall's condition} if $\abs{N(I)}\geq \abs{I}$ for every independent set~$I$ of~$G$. Observe that this is equivalent to $\abs{N(S)}\geq\abs{S}$ for not necessarily independent sets $S$ (here, $N(S)$ is defined to be the set of all vertices connected by an edge to $S$; in particular,~$N(S)$ and $S$ are not necessarily disjoint). Note that neither of these conditions is sufficient to conclude the existence of a perfect matching in a general (non-bipartite) graph~$G$. However, to every graph~$G$, we can associate a bipartite graph $B(G)$ called the \defn{bipartite double cover}, whose vertex set consists of two disjoint copies of $V(G)$, and in which two vertices across the disjoint copies are connected if and only if the corresponding vertices are connected in~$G$. If~$G$ satisfies Hall's condition, then so does $B(G)$, and hence $B(G)$ contains a perfect matching by Hall's theorem. A perfect matching in $B(G)$ corresponds to a collection of (potentially odd) cycles and edges spanning $G$. This fact is a key ingredient that allows us to go from Hall's condition in general graphs to finding regular subgraphs as outlined earlier.

In this section, our goal is to prove the following, which is an essential ingredient in the proof of \cref{lem:path-decomp} we present in the next section. 

\begin{prop}\label{thm:hall}
    Let $G$ be a spanning subgraph of $C_n^k$ such that $\delta(G) \geq k + 1$. Then $G$ satisfies Hall's condition. 
\end{prop}
\begin{proof}
    We assume a cyclic order on the vertices of $G$ inherited from the cycle $C_n$. Let $S \subseteq V(G)$ be a fixed independent set of $G$. For every vertex $v$ of $S$, let $\pi(v)$ be the interval of $C_n^k$ comprising $v$ and the $k$ neighbors of $v$ in $C_n^k$ to the right of $v$. Now, we define the following sequence of vertices of $S$:
    \begin{itemize}
        \item Let $u_0$ be a vertex of $S$ chosen arbitrarily.
        \item For $i \geq 1$, let $u_{i+1}$ be the first vertex of $S$ which is strictly to the right of $\pi(u_i)$. 
        \item If $u_{i+1} = u_j$ for some $0 \leq j \leq i$, end the sequence. 
    \end{itemize}

Since $S$ contains a finite number of vertices, the sequence terminates at some $u_{i+1}$. Let $Q = \{u_j, u_{j+1}, \hdots, u_i\}$. Consider the subgraph $G[S]$ together with the cyclic order inherited from $G$. Since $u_{i+1} = u_j$, there is an integer $t\geq 1$ such that the multiset union $\bigcup_{u \in Q}(\pi(u) \cap S)$ consists of exactly $t$ copies of $S$. 
 
For $v \in V(G)$, let $N_L(v)$ denote the neighborhood of $v$ to the left of $v$ in the cyclic order. Let $N \coloneqq \bigcup_{u \in Q} N_L(u)$ and let $H$ be the bipartite graph with $V(H) = Q \sqcup N$ where $u\in Q$ and $v\in N$ are connected if and only if $v\in N_L(u)$.

\begin{claim}
    The number of edges in $H$ is at least $t\abs{S}$. 
\end{claim}

\begin{proofofclaim}
    By the definition of $t$, every vertex of $S$ appears $t$ times in the multiset union $\bigcup_{u \in Q} \pi(u)$.
    Let $u \in Q$. 
    Since $d(u) \geq k + 1$ and $u$ is anticomplete to $(\pi(u) \cap S)\setminus\{u\}$, it follows that $u$ has at least $\abs{\pi(u)\cap S}$ neighbors to its left. 
    We conclude that 
    \[\abs{E(H)} = \sum_{u\in Q}\abs{N_L(u)}\geq \sum_{u \in Q} \abs{\pi(u)\cap S}=t\abs{S}.\qedhere\]
\end{proofofclaim}

\begin{claim}
    Every vertex of $N$ has degree at most $t$ in $H$.
\end{claim}

\begin{proofofclaim}
    Let $x$ be an arbitrary vertex of $N$ of some degree $\ell$ in $H$. 
    Let $q_1,\ldots,q_\ell$ be the neighbors of $x$ in $H$, appearing in $C_n$ to the right of $x$ in this order. 
    We have $x \in \bigcap_{i=1}^\ell N_L(q_i)$. 
    Observe that the interval from $x$ to $q_\ell$ has length at most $k$.
    So, $q_\ell \in \pi(q_i)$ for every $1 \leq i \leq \ell$. Furthermore, $q_\ell \in Q \subseteq S$ and every vertex of $S$ appears in exactly $t$ intervals $\pi(u)$ for $u \in Q$.
    It follows that $\ell \leq t$. 
\end{proofofclaim}

Since there are at least $t\abs{S}$ edges in $H$ and every vertex of $N$ has at most $t$ neighbors in~$H$, it follows that $\abs{N}\geq\abs{S}$. 
In particular, since $N \subseteq N(S)$, we have $\abs{N(S)}\geq\abs{S}$. 
As $S$ was chosen arbitrarily, we conclude that $G$ satisfies Hall's condition. 
\end{proof}

\section{Path decomposition}\label{sec:path-decomp}

In this section, we will prove \cref{lem:path-decomp}. 

\subsection{Well-distributed matching}

We will make use of the following lemma that extracts a simple pseudorandomness condition from the output of a random matching algorithm. Such extensions of the nibble method go back to work of Alon and Yuster \cite{alon2005hypergraph}, and have been developed further in \cite{ehard2020pseudorandom}. For the convenience of the reader, we choose to give a self-contained treatment here as many of the technicalities simplify in the context of ($2$-uniform) graphs.

\begin{lemma}\label{lem:nice-matching}
    Let $d$ be sufficiently large.
    Let $H=(U\sqcup V,E)$ be a bipartite graph of minimum degree at least $(1-\eta)d$ and maximum degree at most $(1+\eta)d$, for some $\eta\in(0,1/2)$.
    Let $\mathcal P$ be any collection of subsets of $U$ and $V$ of size between $(1-\eta)d$ and $(1+\eta)d$ (which must each be subsets of either $U$ or $V$) which cover each vertex of $U\sqcup V$ at most twice. 

    One can find a matching in $H$ which covers at least a $1-6\max(\eta,d^{-6/25})$-fraction of every set in $\mathcal P$.
\end{lemma}

We will prove this by iterating the following more technical lemma.

\begin{lemma}\label{lem:nice-matching-iter}
    In the setting of \cref{lem:nice-matching}, if $d$ exceeds some absolute constant and $\eta\geq d^{-1/2}$, one can find some $d'\in(\frac 9{10}d,d-\frac1{20}\sqrt d)$, some $\eta'\geq (d')^{-1/2}$, and some matching $E_1$ in $H$ which satisfy the following properties.
    Let $H'$ be the graph formed from $H$ by removing all vertices incident to some edge of $E_1$.
    Then  
    \begin{itemize}
        \item $d'\eta'+(d')^{19/25}\leq d\eta+d^{19/25}$,
        \item for each vertex $v\in V(H')$, we have $d_{H'}(v)\in[(1-\eta')d',(1+\eta')d']$ and
    
        \item for each set $S\in\mathcal P$, we have $\abs{S\cap V(H')}\in[(1-\eta')d',(1+\eta')d']$.
    \end{itemize}
\end{lemma}

\begin{proof}[Proof of \cref{lem:nice-matching} given \cref{lem:nice-matching-iter}]

    We may assume $\eta\geq d^{-1/2}$ without loss of generality.
    Let $\eta_0=\eta$ and $d_0=d$.
    Iteratively apply \cref{lem:nice-matching-iter} to find some sequences $(d_i),(\eta_i),(H_i)$ for which $\eta_i\geq d_i^{-1/2}$ for each $i$, the sequence $r_i:=d_i\eta_i+d_i^{19/25}$ is nonincreasing, each $H_i$ is formed from $H_{i-1}$ by removing the vertices incident to some matching, and
    \begin{itemize}
        \item for each vertex $v\in V(H_i)$ we have $d_{H_i}(v)\in[(1-\eta_i)d_i,(1+\eta_i)d_i]$, and
        \item for each set $S\in\mathcal P$ we have $\abs{S\cap V(H_i)}\in[(1-\eta_i)d_i,(1+\eta_i)d_i]$.
    \end{itemize}
    Each $H_i$ is formed from $H$ by removing the vertices incident to some matching.
    We are forbidden from applying \cref{lem:nice-matching-iter} only when $\eta_i\geq1/2$ or $d_t<C$ for some absolute constant~$C$ at which \cref{lem:nice-matching-iter} breaks (which we can assume exceeds $1000$).
    In fact, one of these conditions will eventually occur, since $\lfloor d_{i+1}\rfloor\leq \lfloor d_i-\frac1{20}\sqrt d_i\rfloor\leq \lfloor d_i\rfloor-1$ if $d_i>1000$. 

    So, we can let $t<\infty$ be such that $\eta_t>1/2$ or $d_t<C$, and let $M$ be the matching of $H$ for which $H_t$ was formed by removing the vertices incident to $M$.
    Assume $d>C^2$.
    For each $S\in\mathcal P$ we have
    \[\abs{S\cap V(H_t)}\leq (1+\eta_t)d_t.\]
    If $d_t\leq C$ then this is at most $2C\leq 4d\eta$; otherwise we have $d_t>1000$ and $\eta_t>1/2$, and so
    \[\abs{S\cap V(H_t)}\leq(1+\eta_t)d_t<3\eta_td_t<3r_t<3r_0= 3d\eta+3d^{19/25}\leq 6d\max(\eta,d^{-6/25}).\]
    In either case, we obtain the result.
\end{proof}

\begin{proof}[Proof of \cref{lem:nice-matching-iter}]

    We will use a nibble procedure based on the Lov\'asz local lemma.

    Pick some $\zeta\in(1/d,1/2)$ and $\delta\in(0,1)$ to be specified later, and let $p=\zeta/d$.
    Let $E_1$ be a subset of $E$ in which each element is chosen independently with probability $p$.
    Say a vertex is \emph{chosen} if it is an endpoint of some edge in $E_1$. We define three types of events:
    \begin{itemize}
        \item For any pair $\{e_1,e_2\}$ of incident edges, let $\mathcal E_{e_1,e_2}$ be the event that both $e_1$ and $e_2$ lie in $E_1$.
        (Call these events \emph{type I}.)
    
        \item For any vertex $v$, let $\mathcal E_v$ be the event that the number of non-chosen neighbors of $v$ is not between $(1-\zeta(1+\delta))(1-\eta)d$ and $(1-\zeta(1-\delta))(1+\eta)d$.
        (Type II)
    
        \item For each set $S\in\mathcal P$, let $\mathcal E_S$ be the event that the number of chosen elements of $S$ lies not between $(1-\zeta(1+\delta))(1-\eta)d$ and $(1-\zeta(1-\delta))(1+\eta)d$.
        (Type III)
    \end{itemize}

    Each vertex $w$ is chosen with probability exactly $1-(1-p)^{d_w}$, which (because of the inequalities $(1-\eta)d\leq d_w\leq (1+\eta)d$) can be computed to lie in
    \[\left[\zeta(1-\eta)-\zeta^2,\zeta(1+\eta)+\zeta^2\right].\]
    For a set $T$ contained in $A$ or $B$, let $X_T$ be the number of chosen elements of $T$, and let $\mu_T=\EE X_T$ and $\rho_T=\frac{\mu_T}{\zeta\abs{T}}-1$.
    The above computation implies $\abs{\rho_T}\leq\zeta+\eta$.
    For each such~$T$, the number of chosen elements of $T$ is a sum of independent $\{0,1\}$ random variables.
    As a result, for any $\delta\in(\abs{\rho_T},1)$, a Chernoff bound implies
    \begin{align*}
        \Pr\big[\abs{X_T-\zeta\abs{T}}\geq\delta\zeta\abs{T}\big]
        &\leq \Pr\big[\abs{X_T-\mu_T}\geq(\delta-\abs{\rho_T})\zeta\abs{T}\big]\\
        &\leq 2\exp\left(-\frac1{10}(\delta-\abs{\rho_T})^2\zeta\abs{T}\right)\\
        &\leq 2\exp\left(-\frac1{10}(\delta-\eta-\zeta)^2\zeta\abs{T}\right).
    \end{align*}
    
    Next, we claim that each of the above events is mutually independent of all but at most~$40d^3$ other events.
    Denote by $R_{e_1,e_2} = \{e_1,e_2\}$, $R_{v} = \{vw : w \in N(v)\}$ and $R_S = \{vw : v \in S\}$ the \emph{witnessing} sets of edges whose random choices determine the events $\mathcal E_{e_1,e_2}$, $\mathcal E_v$, and $\mathcal E_S$, respectively.
    By assumption, we can bound $\abs{R_v} \leq  (1+\eta)d$ and $\abs{R_S} \leq \abs{S} (1+\eta)d \leq (1+\eta)^2d^2$.
    On the other hand, a fixed edge $uv$ is contained in at most $\abs{N(u)}+\abs{N(v)} \leq 2(1+\eta)d$ witnessing sets of type I, at most $2$ witnessing sets of type II, namely $\mathcal E_u$ and $\mathcal E_v$, and at most~$4$ witnessing sets of type III, since every vertex is covered by at most $2$ sets of $\mathcal P$.
    Since two events are independent unless their respective witnessing sets share an edge, every event is mutually independent of all but at most $(1+\eta)^2 d^2 (2(1+\eta)d + 2 +4 )\leq 40d^3$ events, as desired.

    We may therefore apply the Lov\'asz local lemma to find that, with positive probability, no event of type I, II, or III occurs, as long as 
    \[
        \zeta^2\leq \frac{1}{100d}\,,\qquad
        \delta>\eta+\zeta\,,\qquad\text{and}\qquad 
        (\delta-\eta-\zeta)^2\zeta d\geq10\log d.
    \]
    This allows us to replace $H$ with some induced subgraph $H'$ with degrees lying between $(1-\zeta+\delta\eta\zeta)d\pm(\eta-\eta\zeta+\delta\zeta)d$ and replace $\mathcal P$ with a corresponding partition of $V(H')$ with parts of sizes within $(1-\zeta+\delta\eta\zeta)d\pm(\eta-\eta\zeta+\delta\zeta)d$. 
    Let
    \[d':=(1-\zeta+\delta\eta\zeta)d,\quad \eta':=\frac{\eta-\eta\zeta+\delta\zeta}{1-\zeta+\delta\eta\zeta}.\]
    What remains is to choose $\delta$ and $\zeta$. 
    Set $\zeta=1/(10\sqrt d)$ and 
    \[
        \delta=\eta+\zeta+\sqrt{\frac{10\log d}{\zeta d}}=\eta+\zeta+10\frac{\sqrt{\log d}}{d^{1/4}}.
    \]
    For $d$ larger than some absolute constant we have $\delta<1$, and we have $\eta<1/2$ by assumption, so $\frac 9{10}d<d'<d-\frac1{20}\sqrt d$ and
    \[\eta'=\frac{\eta-\eta\zeta+\delta\zeta}{1-\zeta+\delta\eta\zeta}>\frac{\eta-\eta\zeta+\delta\eta^2\zeta}{1-\zeta+\delta\eta\zeta}=\eta.\]
    Moreover, we have
    \[d'\eta'=(\eta-\eta\zeta+\delta\zeta)d=\left(\eta+\zeta^2+\sqrt{\frac{10\zeta\log d}d}\right)d>d\eta,\]
    and so $(d')^{1/2}\eta'\geq (d')^{-1/2}(d\eta)\geq d^{1/2}\eta\geq 1$. 
    Finally we have
    \begin{align*}
        d'\eta'-d\eta
        &=(\eta-\eta\zeta+\delta\zeta)d-d\eta=(\delta-\eta)\zeta d\\
        &=\zeta^2 d + \sqrt{10\zeta d\log d}\\
        &=\frac1{100}+d^{1/4}(\log d)^{1/2}\\
        &\leq \frac{19}{500}d^{13/50}=d^{19/25}\frac{19}{25}\cdot\frac1{20}d^{-1/2}\\
        &\leq d^{19/25}\left(1-\left(1-\frac1{20}d^{-1/2}\right)^{19/25}\right)\\
        &=d^{19/25}-\left(d-\frac1{20}\sqrt d\right)^{19/25}\leq d^{19/25}-(d')^{19/25},
    \end{align*}
    where we have used the fact that $d$ is sufficiently large in the fourth line and the inequality $\eps p\leq1-(1-\eps)^p$ for $0<p,\eps<1$ to go from the fourth line to the fifth. 
    The result follows.    
\end{proof}

\subsection{Proof of the path decomposition lemma}

\begin{proof}[Proof of \cref{lem:path-decomp}]

    We begin by using the robust Hall condition to find a regular spanning subgraph $G_0$ of $G$.
    Let $\ell=2\lfloor \alpha k/2\rfloor$.

    \begin{claim}
        There exists a multiset $E_0$ of edges of $G$ for which $G_0 = (V(G),E_0)$ is $\ell$-regular in which every edge has multiplicity at most $2$.
    \end{claim}
    \begin{proofofclaim}
        Hall's condition in a graph $G$ is equivalent to (bipartite) Hall's condition in the bipartite $2$-lift (also called a bipartite double-cover) of $G$ (recall the discussion in \cref{sec:Hall}). 
        We conclude that any graph $G$ satisfying Hall's condition has a cycle factor (a component of which is allowed to be a single edge, thought of as a $2$-cycle). 
        By \cref{thm:hall}, we can thus find $\ell/2$ edge-disjoint cycle factors of $G$ by peeling them off one-by-one to form successive subgraphs $G'$ of $G$ until the condition $\delta(G')>k$ fails to hold. 
        Let $E_0$ be the multiset of edges contained in these $\ell/2$ cycle factors, where we include an edge which appears in a $2$-cycle in some cycle factor twice. 
        In this way, the (not necessarily simple) graph $G_0=(V(G),E_0)$ is~$\ell$-regular.
    \end{proofofclaim}

    Let $c=\lfloor k^{1/12}/100\rfloor$ and $r=k^{-1/12}$.
    By repeatedly applying \cref{lem:distributed-partition} to $G_0$, we can partition
    \[V(C_n)=V_1\sqcup\cdots\sqcup V_c\]
    in such a way that each $V_i$ is $(1/c,r)$-distributed for $G_0$ (here we use $r/c\geq k^{-1/5}$).
    Write $\eps=1/c$ and $d=\ell/c$ and $\eta=2rk^{4/5}d^{-1}<400\alpha^{-1}k^{-1/5}<1/2$.
    Since each $V_i$ is $(\eps,r)$-distributed, we have
    \begin{equation}\label{eq:Vi-distr}
        \abs[\big]{\abs{I\cap V_i}-\eps\abs{I}}\leq rk^{4/5}
    \end{equation}
    for each interval $I$ of length at most $2k+1$. 
    Moreover, the vertex degrees of the bipartite graphs $H_i:=G_0[V_i\sqcup V_{i+1}]$ are each bounded between $(1-\eta)d$ and at most $(1+\eta)d$. 
    Indeed, let $x\in V_i$.
    Since $V_{i+1}$ is $(\eps,r)$-distributed for $G_0$ and $x$ has degree exactly $\ell$ in $G_0$, we have
    \begin{equation}\label{eq:Vi-deg-distr}
        \abs*{\deg_{H_i}(x)-d}=\abs[\big]{\abs{V_{i+1}\cap N_{G_0}(x)}-\eps\abs{N_{G_0}(x)}}\leq rk^{4/5}<\eta d.
    \end{equation}
    
    Let $\mathcal P$ be a collection of intervals of length between $(1-\eta/2)\ell$ and $(1+\eta/2)\ell$ which cover $V(C_n)$ and cover each vertex of $C_n$ at most twice.   
    For each $i$, let $\mathcal P_i:=\{I\cap V_i\colon I\in \mathcal P\}$.
    The bound \eqref{eq:Vi-distr} implies that, for each $I \in \mathcal P$,
    \[\abs[\big]{\abs{I\cap V_i}-d}\leq\abs[\big]{\abs{I\cap V_i}-\eps\abs{I}}+\abs[\big]{\eps\abs{I}-d}\leq rk^{4/5}+\frac\eta2d=\eta d.\]
    By \eqref{eq:Vi-distr} and the above, the bipartite graph $H_i$ along with the collection $\mathcal P_i\sqcup\mathcal P_{i+1}$ satisfies the preconditions of \cref{lem:nice-matching}.
    As a result, there exists a matching $\Sigma_i$ in $H_i$ which covers at least a $(1-6\max(\eta,d^{-6/25}))$-fraction of each set in $\mathcal P$. 
    Let $U_i$ and $U_{i+1}'$ be the set of vertices in~$V_i$ and $V_{i+1}$, respectively, not covered by $\Sigma_i$.
    We compute, assuming that $k$ is larger than some absolute constant,
    \[\eta=2rck^{4/5}\ell^{-1}>\frac1{100}\alpha^{-1}k^{-1/5}>2\alpha^{-1}k^{-11/50}>2(\alpha k^{11/12})^{-6/25}>d^{-6/25}.\]

    The graph $(V(C_n),\bigcup E(\Sigma_i))$ is a spanning linear forest of $G$. The multiset of endpoints of this spanning linear forest is exactly
    \[\mathcal U:=(V_1\sqcup U_1)\cup (V_c\sqcup U_c')\cup\bigcup_{i=1}^{c-1}(U_i\sqcup U_i').\]
    It remains to show that $\mathcal U$ is locally sparse.

    First, we note that $V_1$ and $V_c$ are both $(\eps,r)$-locally sparse, and so each intersects each interval $I$ of length at most $2k+1$ in at most $\eps\abs{I}+\frac1{10}k^{4/5}$ elements.
    Next, each interval $I$ in~$C_n$ is covered by at most $\abs{I}/((1-\eta/2)\ell)+2\leq 2\abs{I}/\ell+2$ elements of $\mathcal P$. We conclude that
    \begin{align*}
        \abs*{\left(U_1\cup U_c'\cup \bigcup_{i=2}^{c-1}(U_i\sqcup U_i')\right)\cap I}
        &\leq \left(2\frac{\abs{I}}\ell+2\right)\max_{J\in\mathcal P}\abs*{\left(U_1\cup U_c'\cup \bigcup_{i=2}^{c-1}(U_i\sqcup U_i')\right)\cap J}\\
        &\leq \left(2\frac{\abs{I}}\ell+2\right)\cdot 2c\cdot 6\eta(1+\eta/2)d\\
        &\leq 25(\abs{I}+\ell)\eta=25\abs{I}\eta+50rck^{4/5}\leq 25\abs{I}\eta+\frac12k^{4/5}.
    \end{align*}
    Hence $\mathcal U$ is $(20\eps+30\eta,1/2)$-locally sparse. Since $\eta<400\alpha^{-1}k^{-1/5}<\eps<110k^{-1/12}$, this finishes the proof.  
\end{proof}

\section{Dirac-type thresholds for clique containment}\label{sec:dirac-clique-proof}

Having concluded the proof of \cref{thm:main}, we now move on to the other extremal questions in $C_n^k$ discussed in \cref{sec:other-q}. 
In this section, we prove \cref{thm:dirac-clique}, answering the Dirac-type problem in $C_n^k$ for clique containment: what minimum degree is necessary to guarantee that a subgraph of $C_n^k$ has $K_t$ as a subgraph?
First, we present a construction to show that the bound in \cref{thm:dirac-clique} is tight.

\begin{lemma}\label{lem:dirac-clique-construction}
    Let $n>2k+\frac{2k}{t-2}$ be divisible by $\frac k{t-2}$.
    Then there exists a $K_t$-free subgraph of~$C_n^k$ with minimum degree $(2-\frac1{t-2})k+1$.
\end{lemma}

\begin{proof}
    Indeed, partition the vertex set of $C_n^k$ into intervals of size $k/(t-2)$, and let $G$ be the graph obtained from $C_n^k$ by deleting all edges that are within the same interval.
    Since at most~$\frac k{t-2}$ edges incident to each vertex have been deleted, it follows that $\delta(G) \geq (2-\frac1{t-2})k+1$.
    Moreover, $G$ does not contain $K_t$ as a subgraph.
    Indeed, such a $K_t$ would need to have vertices in $t$ intervals.
    Among any $t$ intervals, one can find two intervals with at least $t-2$ intervals between them in either direction, since there are at least $2t-2$ intervals in total. However, any two vertices in these two intervals are at distance strictly exceeding $k$ along $C_n$, and so they cannot be connected by an edge of $G\subset C_n^k$.  
\end{proof}

It is worth noting that this construction also gives a lower bound for the value of $\alpha$ in \cref{qn:Kt-factor}.

\begin{corollary}\label{cor:dirac-factor-construction}
    Let $n>2k+\frac{2k}{t-2}$ be divisible by $(t-1)\frac k{t-2}$.
    Then there exists a $K_{t-1}$-factor-free subgraph of $C_n^k$ with minimum degree $(2-\frac1{t-2})k$.
\end{corollary}

\begin{proof}
    Assume that $n = m k/(t-2)$ with $m$ divisible by $t-1$.
    Partition the vertex set of $C_n^k$ into intervals $I_1,\dots,I_m$ of size
    \begin{align*}
        |I_\ell| = \begin{cases}
              k/(t-2) + 1 & \text{if } \ell \equiv 1 \mod t-1\,, \\ 
            k/(t-2) - 1 & \text{if } \ell \equiv -1  \mod t-1\,, \\ 
              k/(t-2) & \text{otherwise}.
        \end{cases}
    \end{align*}
    Let $G$ be the graph obtained from $C_n^k$ by deleting all edges that are within the same interval.
    It follows that $\delta(G) \geq (2-\frac1{t-2})k$.
    Moreover, $G$ does not contain a $K_{t-1}$-factor as a subgraph, since every copy of $K_{t-1}$ is contained in $t-1$ consecutive intervals.
\end{proof}

We now turn to the proof of the first part of \cref{thm:dirac-clique}, namely that if $G$ is a spanning subgraph of $C_n^k$ with minimum degree strictly exceeding $(2-\frac1{t-2})k+1$, then $G$ has $K_t$ as a subgraph. We begin by defining some notation and terminology. For two vertices $u, v$ of $G$, we denote by $I(u, v)$ the shorter interval of $C_n$ from $u$ to $v$ (inclusive). For a vertex $u$ of $G$, we denote by $I_k(u)$ the interval of $C_n$ consisting of $N_{C_n^k}[u]$. For an edge $uv$ of $G$, the \emph{length} of $uv$ is the distance from $u$ to $v$ in $C_n$. We also fix a mapping of the vertices of $C_n^k$ to the elements of $\ZZ/n\ZZ$, so that for two vertices $u, v$ of $G$, the number of vertices in $I(u, v)$ is $v-u + 1$.

Observe that every complete subgraph $H$ of $G$ contains a unique edge $uv \in E(H)$ such that $V(H) \subseteq I(u, v)$. 
In this case, $uv$ is called the \emph{longest edge} of $H$.

\begin{proof}[Proof of \cref{thm:dirac-clique}]
 
		We argue by induction on $t$, which lets us assume that $t\ge 4$ and that $G$ 
		contains a copy of $K_{t-1}$.
    Let $xz \in E(G)$ be a shortest edge of $G$ such that $xz$ is the longest edge of a $K_{t-1}$ subgraph of $G$. Let $S$ be the set of all vertices of $G$ other than $x$ and $z$ that are in some $K_{t-1}$ subgraph of $G$ with longest edge $xz$, and set $m = \abs{S}$. Fix one $K_{t-1}$ subgraph $Y$ of $G$ with longest edge $xz$, and set $V(Y) = \{y_0 = x, y_1, y_2, \hdots, y_{t-3}, y_{t-2} = z\}$. 
    Observe that $\{y_1, \hdots, y_{t-3}\} \subseteq S$ and that $S \subseteq I(x, z)$. Let $\delta = \delta(G)$ be the minimum degree of $G$. 
    
    \begin{claim}\label{cl1}
    	$\delta \leq 2k + 1 - \frac{m}{t-3}$. 
    \end{claim}
    \begin{proofofclaim}
    Since $G$ does not contain a $K_t$ subgraph and $S$ is complete to $\{x, z\}$, it follows that $G[S]$ does not contain a $K_{t-2}$ subgraph. By Tur\'an's theorem, $\delta(G[S]) \leq \frac{t-4}{t-3} m$, so some vertex of $S$ has at least $\frac{m}{t-3} - 1$ non-neighbors in $S$. As the maximum possible degree of a vertex in $G$ is $2k$, it follows that $\delta \leq 2k - \left(\frac{m}{t-3} - 1\right)$. This proves \cref{cl1}. 
    \end{proofofclaim}
    
    \begin{claim}\label{cl2}
    $(2k+1 - \delta)(t-1) \geq 2k - m$.	
    \end{claim}
    \begin{proofofclaim}
    	For every vertex $y_i$, we want to count the number of non-neighbors of $y_i$ in $I_k(y_i)$. There are $2k+1$ vertices in $I_k(y_i)$ and $y_i$ is adjacent to at least $\delta$ of them, so $y_i$ has at most $2k+1 - \delta$ non-neighbors in $I_k(y_i)$. Summing over every vertex $y_i$, we can write this as the following: 
    	\begin{equation}\label{eq1}
    	(2k+1 - \delta)(t-1) \geq \sum_{i=0}^{t-2} \abs{I_k(y_i) \setminus N(y_i)}.
        \end{equation}
    	
    Observe that $v \in I_k(y_i)$ if and only if $y_i \in I_k(v)$. Therefore, instead of counting the non-neighbors of every vertex $y_i$ in $I_k(y_i)$, we can consider each vertex $v$ of $G$ and count the non-neighbors of $v$ among the vertices $y_i$ that appear in $I_k(v)$. Namely,
    \begin{equation*}
    		\sum_{i=0}^{t-2} |I_k(y_i) \setminus N(y_i)| = \sum_{v \in V(G)} |(I_k(v) \cap Y) \setminus N(v)|\,.
    \end{equation*}
    
    	Which vertices of $G$ are guaranteed to have nearby non-neighbors in $Y$? We will restrict our attention to vertices in $I_k(x) \cap I_k(z)$, observing that this still gives a lower bound: 
    	
    	\begin{equation}\label{eq2}
    		\sum_{i=0}^{t-2} |I_k(y_i) \setminus N(y_i)| \geq \sum_{v \in I_k(x) \cap I_k(z)} |(I_k(v) \cap Y) \setminus N(v)|\,.
    \end{equation}
    
    	Let $v$ be a vertex of $I_k(x) \cap I_k(z)$. Since $xz \in E(G)$, it follows that $I(x, z) \subseteq I_k(v)$, so in particular $Y \subseteq I_k(v)$.
        Note that $v$ is not adjacent to every vertex of $Y$.
        If $v \in Y$, this is evident as $v$ is not adjacent to itself.
        If $v$ is not in $Y$, it cannot be adjacent to all of $Y$, since in that case $\{v\} \cup Y$ would induce a copy of $K_t$ in $G$. Therefore, $| (I_k(v) \cap Y) \setminus N(v)| \geq 1$ for every vertex $v \in I_k(x) \cap I_k(z)$.
    	
    	Since $I(x, z) \subseteq I_k(x) \cap I_k(z)$, the previous statement holds for every vertex of $I(x, z)$. Next, we claim that the vertices of $I(x, z) \setminus \{x, z\}$ for which this lower bound is tight are in $S$. Suppose that $v$ is a vertex of $I(x, z) \setminus \{x, z\}$ such that $|(I_k(v) \cap Y) \setminus N(v)| = 1$. If the non-neighbor of $v$ in $Y$ is $x$, then $(Y \setminus \{x\}) \cup \{v\}$ is a $K_{t-1}$ subgraph of $G$ with longest edge strictly contained in $I(x, z)$, contradicting that $xz$ is the shortest longest edge of a $K_{t-1}$ subgraph of $G$. Therefore, $v$ is adjacent to $x$. Similarly, $v$ is adjacent to $z$, so the non-neighbor of $v$ in $Y$ is $y_i$ for some $1 \leq i \leq t-3$. Now, $(Y \setminus \{y_i\}) \cup \{v\}$ is a $K_{t-1}$ subgraph of $G$ with longest edge $xz$, so $v \in S$, as desired. It follows that $|(I_k(v) \cap Y) \setminus N(v)| \geq 2$ for every vertex~$v$ of $I(x, z) \setminus (S \cup \{x, z\})$.  
    	
    	We have thus far argued that 
    	\begin{equation*}
    		\sum_{v \in I_k(z) \cap I_k(x)} |(I_k(v) \cap Y) \setminus N(v)| \geq |I_k(x) \cap I_k(z)| + |I(x, z) \setminus (S \cup \{x, z\})|\,.
    	\end{equation*}
    
    	Since there are $x + k - (z - k) + 1$ vertices in $I_k(x) \cap I_k(z)$ and $(z - x - m - 1)$ vertices in $I(x, z) \setminus (S \cup \{x, z\})$, we obtain that
    	\begin{equation}\label{eq3}
    		\sum_{v \in I_k(x) \cap I_k(z)} |(I_k(v) \cap Y) \setminus N(v)| \geq 2k - m\,.
    	\end{equation}
    
    	Combining equations \eqref{eq1}, \eqref{eq2}, and \eqref{eq3}, we get 
    	$$(2k+1 - \delta)(t-1) \geq 2k-m \,.$$  
    	This proves \cref{cl2}.  
    \end{proofofclaim}
    
    Finally, we can combine the inequalities from \cref{cl1} and \cref{cl2} to obtain
    \begin{equation*}
        2k-(2k+1-\delta)(t-1) \leq m \leq (t-3)(2k+1-\delta)\,.
    \end{equation*}
    For $x = 2k+1-\delta$, this gives 
    \begin{align*}
          && 2k - x(t-1)      &\leq (t-3)x \\
    \iff  && 2k               &\leq (t-3)x + x(t-1) = 2x(t-2) \\
    \iff  && \frac{k}{t-2}    &\leq x  \\
    \iff  && \delta           &\leq 2k+1 - \frac{k}{t-2}\, ,
    \end{align*}
    as desired.
    This completes the proof.
\end{proof}

\section{Tur\'an-type thresholds for clique containment}\label{sec:turan-clique-proof}

In this section, we prove \cref{prop:turan-K3} on the maximum number of edges in a $K_t$-free subgraph of $C_n^k$.

\subsection{Lower bounds}\label{sec:turan-lower}

We begin by proving the lower bound in \cref{prop:turan-K3}. The construction is similar to that in \cref{sec:dirac-clique-proof} for the Dirac-type problem; in fact, this construction also satisfies the relevant minimum-degree condition.

\begin{proof}[Proof of \cref{prop:turan-K3}, lower bound]

    Let $\ell\in\{1,\ldots,k\}$ be a parameter. 
    Since for fixed $k$ we are concerned only with large-$n$ asymptotics, we may assume that $\ell\mid n$.
    
    Construct a subgraph $G$ of $C_n^k$ as follows:
    \begin{enumerate}[(1)]
        \item Begin with the $(t-2)$-power of an $n/\ell$-cycle.
        \item Blow up each vertex into an interval of length $\ell$, obtaining a graph on vertex set $\ZZ/n\ZZ$.
        \item Remove all edges which do not appear in $C_n^k$.
    \end{enumerate}
    For any choice of $\ell\leq\frac n{2t-2}$, the graph obtained in step (1) of the above procedure is $K_t$-free, and thus $G$ is $K_t$-free. 
    It remains to choose the value of $\ell$ for which $G$ has the largest number of edges. 

    The graph obtained in step (2) of the above procedure is regular of degree $2(t-2)\ell$. 
    We will select $\ell$ between $k/(t-1)$ and $k/(t-2)$, so that only the edges of $G$ corresponding to~$E(C^{t-2}_{n/\ell})\setminus E(C^{t-3}_{n/\ell})$ may be removed in step (3). Consider an edge of length $t-2$ in the graph resulting from step (1), and let the intervals of $G$ to which its incident vertices correspond be $[0,\ell)$ and $[(t-2)\ell,(t-1)\ell-1)$. The corresponding edges removed in step (3) are exactly those edges connecting vertices $i$ and $(t-2)\ell+j$ for $0\leq i,j<\ell$ for which $j-i>k-(t-2)\ell$. Therefore, the number of such edges removed is exactly $\binom{(t-1)\ell-k}2$. We conclude that
    \begin{equation}\label{eq:edge-ct}
    e(G)=n(t-2)\ell-\frac{n}\ell\binom{(t-1)\ell-k}2.
    \end{equation}
    Now, set $\beta=((t-2)^2+1)^{-1/2}$ and let $\ell=\floor{\beta k}$.\footnote{This choice of $\ell$ is not always optimal given $k$, but it is within an additive constant of the optimal choice. If one wants to compute the maximum number of edges in a $K_t$-free graph of $C_n^k$ for some \emph{fixed} $k$, one should simply optimize \eqref{eq:edge-ct} over all valid $\ell$.} We claim that the construction given above for this choice of $\ell$ gives the lower bound of \cref{prop:turan-K3}. Indeed, for fixed $t$ as $k$ grows, we have by \eqref{eq:edge-ct} that
    \begin{align*}
    \frac{e(G)}{e(C_n^k)}
    &=\frac1k\left((t-2)\ell-\frac1\ell\binom{(t-1)\ell-k}2\right)\\
    &=(t-2)\beta-\frac1{2\beta}((t-1)\beta-1)^2+O\left(\frac1k\right)\\
    &=t-1-\sqrt{(t-2)^2+1}+O\left(\frac1k\right)\,,
    \end{align*}
    where the last equality follows from a routine, but somewhat tedious, calculation.
\end{proof}

\subsection{An upper bound for triangles}\label{sec:turan-upper}

Our approach towards finding upper bounds is based on the following linear programming problem.

\begin{question}\label{qn:turan-K3-linprog} 
    Let $f\colon\{1,\ldots,k\}\to [0,1]$ be such that, for every multiset $A\subset\{0,1,\ldots,k\}$,
    \begin{equation}\label{eq:lp-A}
    \sum_{\substack{i,j\in A\\i<j}}f(j-i)\leq\floor*{\frac{\abs{A}^2}4}\,.
    \end{equation}
    What is the maximum possible value of $\frac1k\sum_{x=1}^kf(x)$?
\end{question}

\noindent Write $\rho_k$ for the answer to \cref{qn:turan-K3-linprog}.

We connect this question to \cref{prop:turan-K3} via the following lemma, whose proof applies Mantel's theorem.

\begin{lemma}\label{lem:K3-linprog-connection} Every subgraph $G$ of $C_n^k$ without any triangles satisfies $\frac{e(G)}{e(C_n^k)}\leq \rho_k$.
\end{lemma}
\begin{proof} Identify the vertices of $C_n^k$ with $\ZZ/n\ZZ$ in the natural way.
We define a function $f\colon\{1,\ldots,k\}\to [0,1]$ such that $f(i)$ is the proportion of edges $(x,x+i)$ which are contained in $E(G)$. 

We claim that, for any $A\subset\{0,1,\ldots,k\}$, the bound \eqref{eq:lp-A} holds. 
Indeed, consider for each $x\in\ZZ/n\ZZ$ the subgraph of $G$ induced by the vertex set $x+A$. (If $A$ contains some element with multiplicity $m>1$, blow up the corresponding vertex into $m$ vertices.) 
This subgraph is a triangle-free graph on $\abs{A}$ vertices, and so Mantel's theorem implies that it has at most~$\floor{\abs{A}^2/4}$ edges. 
Averaging this bound over $x\in\ZZ/n\ZZ$ gives \eqref{eq:lp-A}.
The result then follows from the fact that $\sum_{x=1}^knf(x)$ is the number of edges of $G$.  
\end{proof}

We believe that the reduction to a linear program described by \cref{qn:turan-K3-linprog,lem:K3-linprog-connection} is tight, in the sense that the function $f$ arising from the construction in \cref{sec:turan-lower} is exactly the optimizer sought in \cref{qn:turan-K3-linprog}. We have verified this computationally for all $k\leq 100$. This function $f$ is roughly given by
\[f(x)=\min\left(\frac{x\sqrt 2}k,2-\frac{x\sqrt2}k\right).\]

We are unable to prove that such a function $f$ indeed answers \cref{qn:turan-K3-linprog}. We will instead use \cref{lem:K3-linprog-connection} to provide the weaker upper bound on $\pi_{\mathrm{cyc}}(K_3)$ claimed in \cref{prop:turan-K3}. This upper bound is a corollary of the following lemma. Its proof uses only those conditions~\eqref{eq:lp-A} with $\abs{A}\leq5$; this is the best bound which can be obtained with those conditions only.

\begin{lemma}\label{lem:rho-computation}
    As $k\to\infty$, we have $\rho_k\leq 3/5+o(1)$.
\end{lemma}

\begin{proof} 
    Recall that $\rho_k$ denotes the answer to \cref{qn:turan-K3-linprog}.
    We observe that ${\abs{\rho_k-\rho_{(1+\eta)k}}\leq \eta}$ for any $\eta>0$. Indeed, if $f_k$ attains the maximum value $\rho_k$, we can extend $f_k$ by zeros to~$\{1,\ldots,(1+\eta)k\}$ to obtain
    \[\rho_{(1+\eta)k}\geq\frac1{1+\eta}\rho_k\geq\rho_k(1-\eta)\geq \rho_k-\eta.\]
    On the other hand, if $f_{k'}$ attains the maximum value $\rho_{(1+\eta)k}$, we can restrict $f_{k'}$ to $\{1,\ldots,k\}$ to obtain
    \[\rho_k\geq\frac1k\sum_{x=1}^kf_{k'}(x)\geq\frac1k\left(\rho_{(1+\eta)k}(1+\eta)k-\eta k\right)\geq\rho_{(1+\eta)k}-\eta.\]
    
    So, it suffices to control $\rho_k$ on a ``dense'' set of integers. We will show that there exists a set $K$ of positive integers which satisfies 
    \begin{enumerate}
        \item $\rho_k\leq 3/5+o(1)$ for $k\in K$; and,
        \item for every positive integer $m$, there exists some $k\in K$ with $k=m(1+o(1))$.
    \end{enumerate}
    (The ``$o(1)$'' expressions in (1) and (2) are as $k\to\infty$.) The result will follow.
    
    We define $K$ to be the set of positive integers $k$ for which the following holds: there exists some probability distribution $\mathcal D_k$ on $\ZZ^5$ for which, if $(X_1,\ldots,X_5)\sim\mathcal D_k$, then each $X_i-X_j$ (for $i\neq j$) has the uniform distribution on $[-k,k]$.
    
    \begin{claim}\label{cl:K-enough} If $k\in K$ then $\rho_k\leq\frac35+\frac3{10k}$.
    \end{claim}
    
    \begin{proofofclaim} 
    
    Let $\mathcal D_k$ be a distribution witnessing $k\in K$. Let $f\colon\{1,2,\ldots,k\}\to [0,1]$ satisfy \eqref{eq:lp-A}, and extend $f$ to a function $\{-k,\ldots,k\}\to[0,1]$ by $f(-i)=f(i)$ and $f(0)=0$. Now, select $A=\{X_1,\ldots,X_5\}$ with $(X_1,\ldots,X_5)\sim\mathcal D_k$. We have
    \[6=\floor*{\frac{5^2}4}\geq\EE\left[\frac12\sum_{u,v\in A}f(u-v)\right]=\frac12\sum_{1\leq i,j\leq 5}\EE[f(X_i-X_j)]=\frac{10}{2k+1}\sum_{x=-k}^kf(x)\,.\]
    Rearranging and using the evenness of $f$ gives the result. \qedhere
    
    \end{proofofclaim}
    
    By \cref{cl:K-enough}, the set $K$ satisfies property (1) above. We now show (2) by exhibiting some particular elements of $K$.
    
    \begin{claim}\label{cl:K-cl} The set $\{2k+1\colon k\in K\}$ is closed under multiplication.
    \end{claim}
    \begin{proofofclaim} Let $k_1,k_2\in K$ and let $\mathcal D_{k_1},\mathcal D_{k_2}$ be distributions witnessing $k_1\in K$ and $k_2\in K$, respectively. Define a distribution $(X_1,\ldots,X_5)\sim\mathcal D$ by taking $(Y_1,\ldots,Y_5)\sim\mathcal D_{k_1}$ and $(Z_1,\ldots,Z_5)\sim\mathcal D_{k_2}$ independently, and letting $X_i=(2k_2+1)Y_i+Z_i$. The result follows.    
    \end{proofofclaim}
    
    \begin{claim}\label{cl:K-base} We have $2,4\in K$.
    \end{claim}
    
    \begin{proofofclaim} 
    
    This is simply a computation. For $k=2$ we take $\mathcal D_2 = (X_1,\dots,X_5)$ to be a uniformly random permutation of $(0,0,1,2,2)$.\footnote{To explain the solution for $k=2$, note that there are $20$ ordered pairs of distinct elements in $\{-2,-1,0,1,2\}$.
    For each distance $d \in [0,4]$, there are $4$ pairs $i,j$ that give $d=|X_i - X_j|$.
    Hence the difference is uniform.}
    
    For $k=4$, we take $\mathcal D_4$ to be a mixture of the following:\footnote{The argument for $k=4$ is similar.
    Denote the three distributions from above by $(X_1^j,\dots,X_5^j)$ with $j \in [3]$.
    For a distance $d \in [0,4]$ and $j \in [3]$, let $P_d^j$ denote the number of distinct pairs~$i,j$ with $d=|X_i^j - X_j^j|$.
    For instance, $P_2^1 = 2$, $P_2^2 = 3$, and $P_2^3 = 0$ for $d=2$.
    So in particular $4 P_2^1 + 4 P_2^2 + P_2^3 =20$.
    Moreover, one can check that the same holds for the other values of~$d$.
    Hence the difference is uniform if the three distributions are combined in the proportion detailed above.}
    \begin{enumerate}[(1)]
        \item Four parts a uniformly random permutation of $(0,0,2,3,4)$.
        \item Four parts a uniformly random permutation of $(0,0,1,3,4)$.
        \item One part a uniformly random permutation of $(0,0,2,4,4)$. \qedhere
    \end{enumerate}

    \end{proofofclaim}
    
    We conclude from \cref{cl:K-cl} and \cref{cl:K-base} that $\frac{5^a9^b-1}2\in K$ for each $a,b\geq 0$. Condition~(2) above, and thus the result, follows from the following claim:
    
    \begin{claim}\label{cl:K-dense}
        For each positive integer $m$, let $u(m)$ be the nearest element of $K_0:=\{5^a9^b\colon a,b\in\ZZ_{\geq 0}\}$ to $m$. We have $\lim_{m\to\infty}\frac{u(m)}m=1$.
    \end{claim}
    
    \begin{proofofclaim}
    
        Since $5^a\neq 9^b$ for integers $a,b\geq 1$, we have that $\log9/\log5$ is irrational.
        Let $\eps\in(0,1/10)$ be arbitrary. We will show that $\abs{u(m)-m}\leq2\eps m$ for all sufficiently large $m$, which will be enough.
        
        By Dirichlet's approximation theorem, we can find positive integers $p$ and $q$ for which $0<\abs{q\log 9/\log 5-p}<\eps$ and $q<\eps^{-1}$. Letting $T=\ceil{\abs{q\log 9/\log5-p}^{-1}}$, we can find for each $r\in[0,1)$ some $0\leq t\leq T$ for which $qt\log9/\log5-r$ is within $\eps$ of an integer. Let $m>9^{qT}$ be any integer and let $r=\log m/\log 5$. We can find integers $0\leq t\leq T$ and $s\geq0$ for which
        \[\frac1{\log 5}\abs{qt\log 9-\log m+s\log 5}=\abs*{qt\frac{\log 9}{\log5}-\frac{\log m}{\log 5}+s}\leq\eps.\]
        Rearranging, we have
        \[\abs{qt\log 9+s\log 5-\log m}\leq \eps\log 5,\]
        and so $m/(9^{qt}5^s)$ lies in the interval $[5^{-\eps},5^\eps]\subset[1-2\eps,1+2\eps]$.
        \qedhere
    \end{proofofclaim}
    This concludes the proof.
\end{proof}

\noindent We remark that the proof of \cref{lem:rho-computation} can be made effective by appealing to Baker's theorem in transcendental number theory. Baker's theorem implies that, in the proof of \cref{cl:K-dense}, we have the bound $T=q^{O(1)}=\eps^{-O(1)}$. Carrying this bound through the rest of the proof gives that $\rho_k\leq 3/5+O((\log k)^{-C})$ for some absolute constant $C>0$.

\bibliographystyle{abbrv}
\bibliography{bib}

\appendix

\section{A counterexample to \texorpdfstring{\cref{conj:geometric}}{Conjecture 1.3} when \texorpdfstring{$d=2$}{d=2}}\label{app:counterexample}

We begin with the following geometric result.

\begin{prop}\label{prop:geometricquestion} 
    For small enough $r>0$, there exists a measurable subset $X$ of the unit torus~$(\RR/\ZZ)^2$ satisfying the following two properties:
    \begin{enumerate}
        \item $\mu(X)<0.499$.

        \item For every $x\in (\RR/\ZZ)^2\setminus X$,
        \[\frac{\mu(X\cap B(x,r))}{\mu(B(x,r))}>0.501.\]
    \end{enumerate}
\end{prop}

The construction underlying \cref{prop:geometricquestion} was suggested by ChatGPT 5.5 Pro, and the example we present in the following pages is a refinement of that suggestion. The verification that it satisfies \cref{prop:geometricquestion} is due to the authors.
We remark that we were unable to determine any prior results or sources on which ChatGPT 5.5 Pro's underlying reasoning may have relied.

Before proving \cref{prop:geometricquestion}, we use it to demonstrate the falsity of \cref{conj:geometric} for $d=2$.

\begin{proof}[Proof that \cref{conj:geometric} fails for $d=2$]
    We will show that, for some small constant $r$, the random geometric graph $T_2(n,r)$ has, with high probability, a $0.5001$-subgraph which is not Hamiltonian. 

    Let $n$ be a large even integer. 
    Sample $G=(V,E)$ according to $T_2(n,r)$, and let $X$ be as in \cref{prop:geometricquestion}.
    Since $\mu(X)<0.499$, we have
    \begin{equation}\label{eq:V-X-small}
        \abs{V\cap X}<\frac12\abs{V}=n/2
    \end{equation}
    with high probability.
    Now, construct a subgraph $H\leq G$ in which those edges without an endpoint in $X$ are removed. The set $V\setminus (V\cap X)$ is an independent set in $H$ which, by \eqref{eq:V-X-small}, contains at least $n/2+1$ vertices. We conclude that $H$ has no perfect matching; in particular $H$ is not Hamiltonian.

    We claim, however, that (with high probability) $\deg_H(v)>0.5001\deg_G(v)$ for each vertex $v\in V$. Indeed, if $v\in X$, then $\deg_H(v)=\deg_G(v)$, so there is nothing to prove. If $v\not\in X$, the neighborhood of $v$ in $H$ is described exactly by
    \[N_H(v)=V\cap X\cap B(v,r).\]
    By (2) of \cref{prop:geometricquestion}, we have $\mu(X\cap B(v,r))>0.501(\pi r^2)$. As $r$ is constant and $n$ is growing, the counts $V\cap B(v,r)$ and $V\cap (X\cap B(v,r))$ thus concentrate heavily around their means. In particular, for $n$ sufficiently large, we have $\abs{V\cap (X\cap B(v,r))}\geq 0.5001\abs{V\cap B(v,r))}$ with probability at least $1-1/n^2$. Taking a union bound concludes the proof.
\end{proof}

The main engine behind the proof of \cref{prop:geometricquestion} is the following technical lemma which provides a set in $\RR^2$ satisfying (essentially) the same two properties.

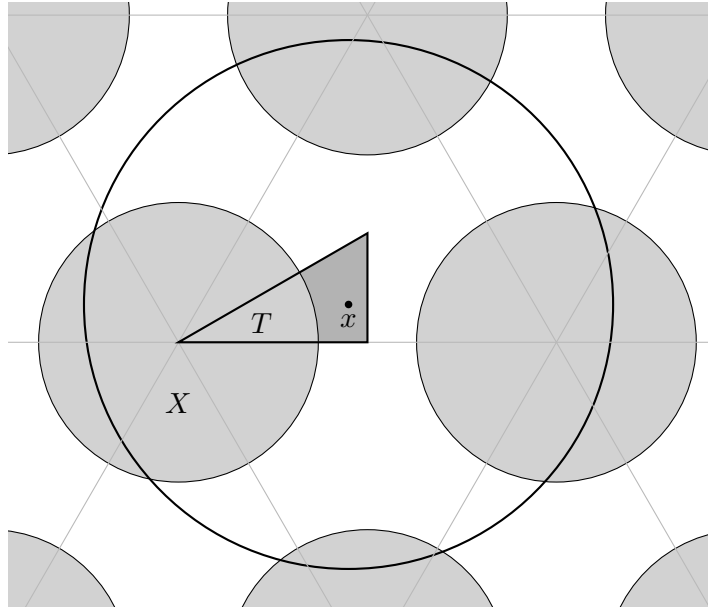
\begin{figure}
\begin{tikzpicture}[scale=5]

\def\rsmall{0.37}
\def\rbig{0.7}
\def\cx{0.45}
\def\cy{0.1}

\pgfmathsetmacro{\rtthree}{sqrt(3)}

\clip (-0.45,-0.7) rectangle (1.4,0.9);

\coordinate (A) at (0,0);
\coordinate (B) at (0.5,0);
\coordinate (C) at (0.5,{sqrt(3)/6});

\begin{scope}
    \clip (A) -- (B) -- (C) -- cycle;
    \fill[gray!60] (-1,-1) rectangle (2,2);
\end{scope}

\fill (\cx,\cy) circle (0.01);
\node at (\cx,\cy) [below] {$x$};

\foreach \i in {-3,...,3} {
    \foreach \j in {-3,...,3} {
        \pgfmathsetmacro{\x}{\i + 0.5*\j}
        \pgfmathsetmacro{\y}{0.5*\rtthree*\j}
        \fill[gray!35] (\x,\y) circle (\rsmall);
        \draw (\x,\y) circle (\rsmall);
    }
}

\draw[thick] (\cx,\cy) circle (\rbig);

\node at (0.22,0.05) {$T$};
\node at (0,-0.16) {$X$};

\foreach \k in {-3,...,3} {
    \pgfmathsetmacro{\y}{0.5*\rtthree*\k}
    \draw[gray!55] (-2,\y) -- (3,\y);
}

\foreach \k in {-3,...,3} {
    \draw[gray!55]
        (-2,{\rtthree*(-2-\k)})
        --
        (3,{\rtthree*(3-\k)});
}

\foreach \k in {-3,...,3} {
    \draw[gray!55]
        (-2,{-\rtthree*(-2-\k)})
        --
        (3,{-\rtthree*(3-\k)});
}

\draw[thick] (A) -- (B) -- (C) -- cycle;

\end{tikzpicture}
\caption{The set $X$ is shaded in light gray and the set $T\setminus X$ is shaded in dark gray. A circle of radius $r=0.7$ is centered at the point $x=(0.45,0.1)\in T\setminus X$.}
\label{fig:triangular-lattice}
\end{figure}

\begin{lemma}\label{lem:circle-inter}
    Set $a=0.37$ and $r=0.7$.
    Let $\Lambda$ be the triangular lattice~$\ZZ(1,0)+\ZZ(1/2,\sqrt3/2)$, and let
    \[X=\bigsqcup_{v\in\Lambda}B(v,a)\]
    be a union of circles of radius $a$. 
    For any $x\not\in X$, the area of the intersection $X\cap B(x,r)$ is at least $0.501(\pi r^2)$.
\end{lemma}

\begin{proof}
    Let $f\colon\RR^2\to[0,1]$ be defined by
    \[f(x)=\frac{\mu(X\cap B(x,r))}{\mu(B(x,r))};\]
    we must show that $f(x)>0.501$ for each $x\not\in X$.
    We observe that
    \[\abs{f(x)-f(y)}=\frac{\abs{\mu(X\cap B(x,r))-\mu(X\cap B(y,r))}}{\pi r^2}\leq\frac{\mu(B(x,r)\triangle B(y,r))}{\pi r^2},\]
    where $\triangle$ denotes the symmetric difference.
    Any line parallel to $x-y$ intersects the symmetric difference $B(x,r)\triangle B(y,r)$ in at most two segments of length each at most $\lVert x-y\rVert$, and so 
    \[\mu(B(x,r)\triangle B(y,r))\leq 2r\lVert x-y\rVert.\]
    We conclude that $f$ is $2/(\pi r)$-Lipschitz.

    The set $X$ admits both translational symmetry of by the lattice $\Lambda$ and dihedral symmetry (via the dihedral group of order $12$) about $(0,0)$. 
    One fundamental domain of this isometry group is the triangle $T$ bounded by the three points $\{(0,0),(1/2,0),(1/2,\sqrt3/6)\}$. The fundamental domain $T$ is pictured in \cref{fig:triangular-lattice}.

    Given these two properties, it is not hard to devise a computer check to verify the correctness of the lemma.
    Our chosen approach is to provide a ``proof by picture:'' a finite set $S$, described in \cref{fig:S-cover}, of points in $T\setminus X$ on which $f(s)>0.501$ such that the union of disks
    \[\bigcup_{s\in S}B\left(s,\frac{f(s)-0.501}{2/(\pi r)}\right)\]
    covers $T\setminus X$. Given explicit formulas for the area of the intersection between two circles, it is not hard to compute $f(s)$ for any fixed $s$. We exhibit the results of this computation in \cref{fig:S-cover} as well.
    
\begin{figure}
\begin{minipage}{0.55\textwidth}
\centering
\begin{tikzpicture}[scale=26]

\def\r{0.37}
\def\R{0.7}

\begin{scope}
\clip (0.3,-0.04) rectangle (0.54,0.32);


\draw (0,0) -- (0.5,0) -- (0.5,0.288675) -- cycle
  (0,0) circle (\r);

\foreach \x/\y/\f in {
{0.335/0.18/0.519111},
{0.347/0.15/0.521093},
{0.365/0.112/0.525619},
{0.369/0.194/0.528549},
{0.375/0.063/0.528383},
{0.386/0.119/0.53061},
{0.39/0.16/0.53132},
{0.392/0.014/0.532733},
{0.42/0.226/0.53975},
{0.425/0.07/0.538115},
{0.432/0.138/0.537686},
{0.436/0.19/0.538162},
{0.463/0.015/0.542647},
{0.483/0.159/0.540848},
{0.484/0.08/0.542685},
{0.49/0.243/0.545276},
}{
  \pgfmathsetmacro{\rad}{(\f-0.501)*pi*\R/2}
  \filldraw[nearly transparent,gray!60] (\x,\y) circle (\rad);
  \draw (\x,\y) circle (\rad);
}
\end{scope}
\end{tikzpicture}
\end{minipage}
\hfill
\begin{minipage}{0.4\textwidth}
    
\begin{tabular}{c|c}
$s$ & $f(s)$ \\
\hline
(0.335, 0.18) & 0.5191 \\
(0.347, 0.15) & 0.5211 \\
(0.365, 0.112) & 0.5256 \\
(0.369, 0.194) & 0.5285 \\
(0.375, 0.063) & 0.5284 \\
(0.386, 0.119) & 0.5306 \\
(0.39, 0.16) & 0.5313 \\
(0.392, 0.014) & 0.5327 \\
(0.42, 0.226) & 0.5398 \\
(0.425, 0.07) & 0.5381 \\
(0.432, 0.138) & 0.5377 \\
(0.436, 0.19) & 0.5382 \\
(0.463, 0.015) & 0.5426 \\
(0.483, 0.159) & 0.5408 \\
(0.484, 0.08) & 0.5427 \\
(0.49, 0.243) & 0.5453 \\
\end{tabular}
\end{minipage}
\caption{The circles centered at points $s\in S$ of radii $\frac{\pi r}2(f(s)-0.501)$, which cover $T\setminus X$, and the corresponding values of $f(s)$, rounded to four decimal places.}
\label{fig:S-cover}
\end{figure}

Having exhibited such a set $S$, there is for every $x\in T\setminus X$ some $s\in S$ for which
    \[\lVert x-s\Vert<\frac{f(s)-0.501}{2/(\pi r)};\]
    from this and the fact that $f$ is $2/(\pi r)$-Lipschitz, we conclude that $f(x)>0.501$ for each such~$x$. 
\end{proof}

\begin{proof}[Proof of \cref{prop:geometricquestion}]
    We will assume $r<10^{-5}$.
    Let $a=37r/70$ and $\ell=10r/7$.
    We construct our set $X$ as follows:
    Let $\Lambda\subset\RR^2$ be the lattice $\ZZ(1,0)+\ZZ(1/2,\sqrt3/2)$, and let $\Lambda_0$ be the intersection of the lattice $\ell\Lambda$ with the square $Q:=[0.0005,0.9995]^2$. Let $X_0=\Lambda_0+B(0,a)$ be the union of the disks of radius $a$ centered at points of $\Lambda_0$ (which are disjoint, since $a<\ell/2$).
    The set $X_0$ lies in the square $[10^{-5},1-10^{-5}]$, and so its image $X_1$ under the projection $\RR^2\to(\RR/\ZZ)^2$ has area equal to that of $X_0$.
    Finally, we set
    \[X=X_1\cup\big((\RR/\ZZ)^2\setminus Q).\]
    
    We now verify (1). 
    The parallelograms $x+[0,\ell](1,0)+[0,\ell](1/2,\sqrt3/2)$ are disjoint for $x\in \Lambda_0$, and they are each contained in the square $[0,1]^2$. 
    As each such parallelogram has area $(\sqrt3/2)\ell^2>1.767r^2$, we conclude that
    \[\abs{\Lambda_0}\leq\frac{1}{1.767r^2}.\]
    We conclude that
    \[\mu(X_1)=\mu(X_0)\leq\frac{\pi a^2}{1.767r^2}<0.497,\]
    and thus
    \[\mu(X)\leq(1-\mu(Q))+\mu(X_1)<0.499.\]
    What remains is to verify (2). Each point $x\in(\RR/\ZZ)^2\setminus X$ may be lifted to some point $\tilde x\in Q\setminus X_0$. Since $r<10^{-4}$, the circle of radius $r$ centered around $\tilde x$ lies entirely within $[0,1]^2$. We conclude that
    \[\mu(X\cap B(x,r))\geq\mu\big((\ell\Lambda+B(0,a))\cap B(\tilde x,r)\big).\]
    Since $\tilde x\not\in \ell\Lambda+B(0,a))$, \cref{lem:circle-inter} says precisely that the area of the intersection $(\ell\Lambda+B(0,a))\cap B(\tilde x,r)$ is at least $0.501(\pi r^2)$. This is enough to prove (2).
\end{proof}

\end{document}